\documentclass{siamltex}

\usepackage{amsmath}
\usepackage{amssymb,latexsym,amsfonts}
\usepackage[mathscr]{eucal}
\usepackage{upgreek}
\usepackage{enumerate}

\usepackage{epsfig}
\usepackage{float}
\usepackage{graphicx}
\usepackage{color}
\usepackage{rotate}

\usepackage{cite}
\makeatletter
\usepackage[curve,arrow,matrix]{xy}
\makeatother

\usepackage{mymacrosSIAM}

\usepackage{algorithmic}

\usepackage{algorithm}

\newcommand{\image}{\mathrm{Image}}

\title{Singular Continuation: Generating Piece--wise Linear Approximations to Pareto Sets via Global Analysis}
\author{Alberto Lovison
\thanks{Dipartimento di Matematica Pura ed Applicata, Via Trieste, 63 35121--Padova, Italy, \texttt{lovison@math.unipd.it}}}
\begin{document}

\maketitle

\begin{abstract}
We propose a strategy for approximating Pareto optimal sets based on the global analysis framework proposed by Smale (Dynamical systems, New York, 1973, pp. 531--544). The method highlights and exploits the underlying manifold structure of the Pareto sets, 
approximating Pareto optima by means of simplicial complexes.  
The method distinguishes the hierarchy between singular set, Pareto critical set and stable Pareto critical set, and can handle the problem of superposition of local Pareto fronts, occurring in the general nonconvex case.
Furthermore, a quadratic convergence result in a suitable set--wise sense is proven and tested in a number of numerical examples. 
\end{abstract}


\pagestyle{myheadings}
\thispagestyle{plain}
\markboth{ALBERTO LOVISON}{GLOBAL ANALYSIS AND MULTIOBJECTIVE OPTIMIZATION}


\section{Introduction}
\subsection{Multiobjective optimization and Pareto optimality}
Multiobjective optimization 
is concerned with the problem of optimizing several functions (or objectives) simultaneously. A precise mathematical statement in an economical framework was first given by 
V. Pareto \cite{Pareto:1896oa,Pareto:1906ya} in the 1880's.
In recent years 
a strong interest has grown,
as a variety of problems in structural mechanics, automotive, aerospace, production planning, environmental policy and many others, involve more than one objective function and different numerical strategies have been developed subsequently  \cite{Miettinen:1999fk}.

In the single objective case, an optimum is defined as a point $x\in W\subseteq \R^n$ where a given function $u:W\to \R$ assumes its maximum, if the maximum exists. 
In multiobjective optimization we consider two or more functions, $u_1,\dots,u_m:W\to\R$, and in all the non trivial cases the optima for one function are distinct from the optima of the remaining ones. A key point is that not only one has to consider the optima of the individual functions, which usually are finite, but also 
one usually finds an infinite number of so--called \emph{non--dominated} points. They are defined precisely as follows.
%
\begin{defin}[Pareto optimality]
Let $W$ be an open subset of $\R^n$, or an $n$--dimensional manifold, and let $u_1,\dots, u_m: W\to \R$ be smooth functions\footnote{We may equivalently refer to a unique smooth \emph{vector function} $u: W\to \R^m$, or \emph{mapping}.}. A point $\bar x\in W$ is called a \emph{non--dominated point}, or a \emph{Pareto optimum}, if there is no $x\in W$ such that $u_i(x)\geq u_i(\bar x)$ for all $i=1,\dots, m$ and $u_j(x) > u_j(\bar x)$ for some $j$. 
If there exists a neighborhood $V\subseteq W$ of $\bar x$ where $\bar x$ is Pareto optimum, then $\bar x$ is called a \emph{local} Pareto optimum.  
\end{defin}
\subsection{The necessity for global representations of the Pareto sets}
As pointed out for instance in \cite{Benson:1997p2942}, the set of Pareto optima is in many cases a large and complicated nonconvex set and most of the existing algorithms, being inspired by local search ideas from traditional linear and nonlinear programming, fail at giving a truly global representation of the set of Pareto optima. 
Also \cite{Das:1998zh} stresses that \emph{``a whole collection of Pareto optimal points, representative of the entire spectrum of efficient solutions''} would be helpful in facilitating design in engineering applications.

Recent multiobjective optimization literature tackled this issue focussing on defining algorithms producing 
even distributions of Pareto points \cite{Das:1998zh,Messac:2004le,Messac:2008fi,
Utyuzhnikov:2009yj} while an alternative philosophy \cite{Rakowska:1991ri,Rakowska:1993mq,Rao:1989lg,Schutze:2005fk} dealt with producing local meshes approximating Pareto sets, relying on continuation (homotopy) strategies. In the recent paper \cite{Pereyra:2009xz},  both topics are addressed. Alternative techniques aiming to approximate the entire optimal set are described in  the recent papers \cite{Gourion:2008rr,Luc:2005sp}, in the survey \cite{Ruzika:2005dz} and in the references therein. 

We want here to highlight a key feature of the Pareto set which makes it, in general nonconvex cases, a complicated set. 
Its complexity is even amplified when the Pareto set is viewed in the output space.
\begin{figure}[htbp]
\begin{center}
\begin{tabular}{cc}
\includegraphics[width=65mm]{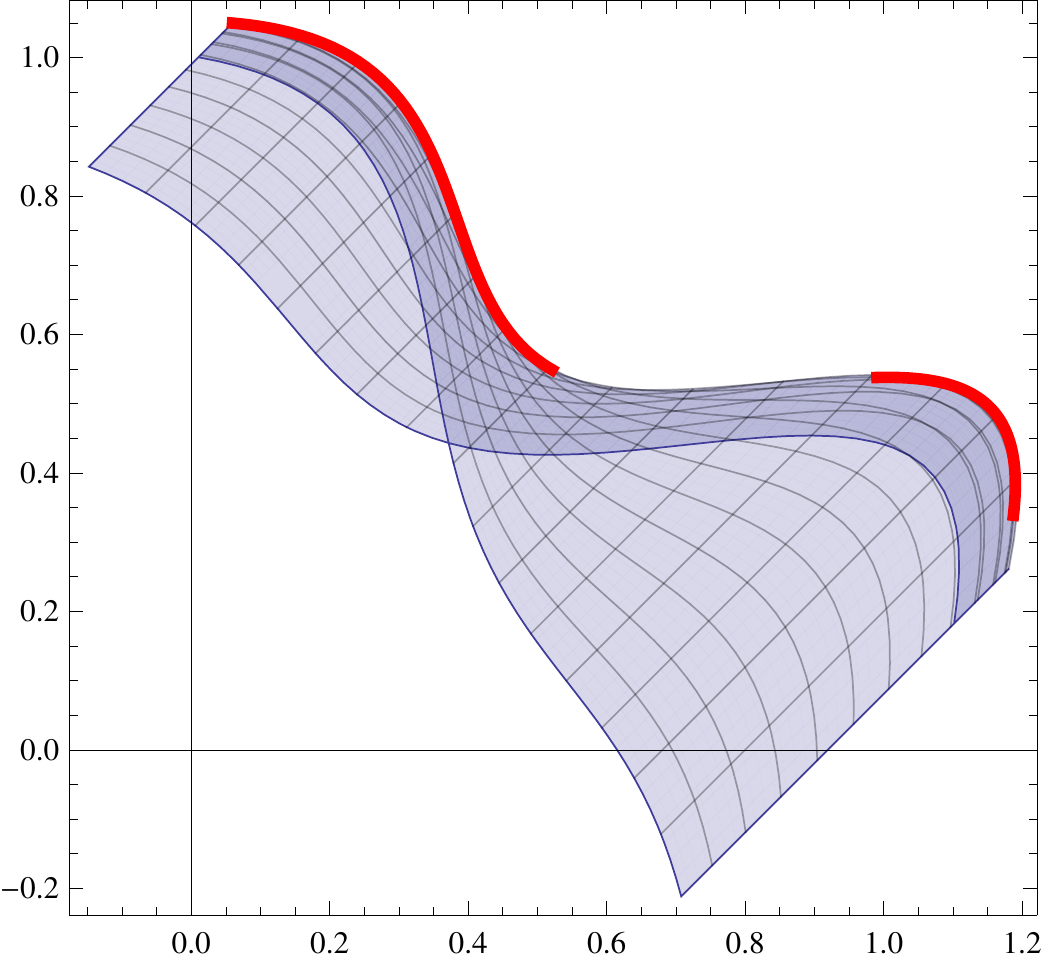} &
\includegraphics[width=55mm]{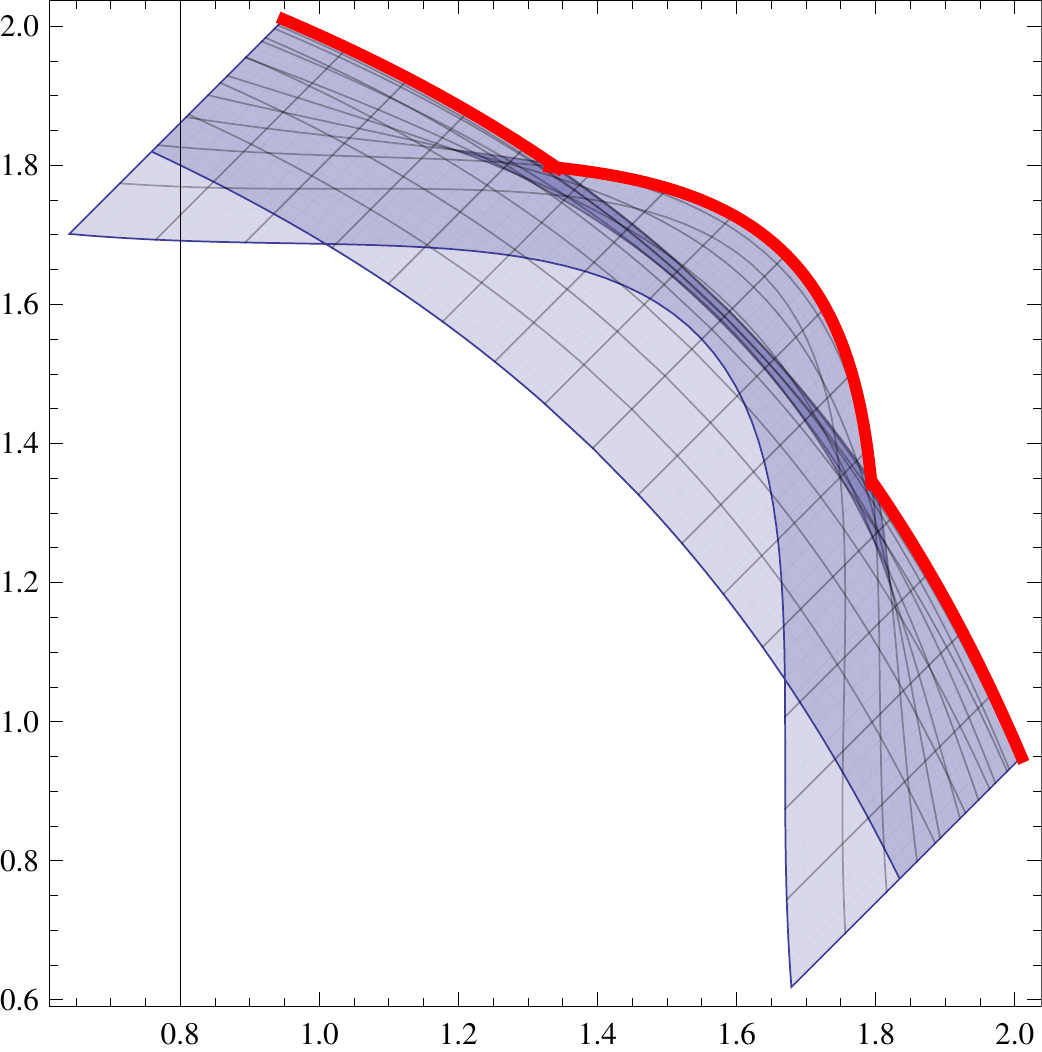} \\
(a) & (b)
\end{tabular}
\caption{Possible problems arising when the objectives are nonconvex functions. (a) The Pareto set is composed by separate branches. (b) A connected global Pareto set is composed by separate local branches crossing each other.}
\label{fig:globalissues}
\end{center}
\end{figure}

Indeed, the set of all global Pareto optima can be disconnected, i.e., composed by separate portions of seemingly smooth surfaces (see figure \ref{fig:globalissues}(a)). Furthermore, even when the image of the set of global Pareto optima is a connected set,  in fact 
it could be composed by 
cutting and sewing together different locally optimal branches (see figure \ref{fig:globalissues}(b)), coming from separate zones of the domain. 
We will illustrate in the sequel that this kind of behavior is not an artifact obtained with unrealistic functions but in a sense represents a typical situation that is not destroyed by slight deformations of the functions. Those situations are persistent, or more technically, \emph{structurally stable}.

We notice that the algorithms mentioned above are expected to work properly only in a local sense, although they are intended to capture some of the global features of the optimal set. Moreover, apart from the homotopy techniques, they are \emph{point--wise} strategies, in the sense that as Pareto optimal set they produce a scatter of points; the evenness of the distribution of points is then estimated on the image space. In some applications those points are joined together in a compound structure, e.g., a Delaunay triangulation, but only a posteriori and in the output space. It should be noticed that because of the effects of mappings described above, defining such a structure from the output space, i.e., joining nearby \emph{values} of optimal points, is subject to failure: the corresponding preimage points, indeed, are not necessarily nearby, or even connected in the way suggested by the positions of their images. Conversely, the images of nearby optimal points are nearby, because of continuity. 

We propose instead that a faithful global representation of the Pareto set in general nonconvex settings is obtained according to the following three steps. Firstly, by shifting the focus from the output space to the input space, secondly by approximating the full set of the local Pareto optima, and thirdly by adopting a \emph{set--wise} standpoint, namely using compound geometrical objects  as simplicial complexes instead of scatters of points.
 The first step unfolds the singularities (branches crossing, cusp points and so on) occurring as an effect of the mapping. 
Indeed, as illustrated in the sequel, the preimage of the Pareto set is non--singular, as it exhibits in general a regular manifold structure.
The second step, because of possible superpositions of local branches, guarantees that every portion of the global Pareto optimal set is represented.   
Thirdly, simplicial complexes, i.e., \emph{meshes}, faithfully reflect the manifold structures and explicitly offer the desired parametrization for each portion of the Pareto set, allowing to perform ``trade--off studies'' among the conflicting objectives. 
Indeed, trade--off studies may be the application of greatest practical importance of multiobjective optimization.  
Nevertheless, from the above discussion it is clear that trying to track the surface of the Pareto set by picking points from the output space, as point--wise strategies are aimed to do, is supposed to work correctly only for limited intervals.  

There are at least two reasons why our program has not yet been pursued in its entirety. 
First of all, in a number of situations, the numerical techniques available in literature are able to build sufficiently faithful representations of the Pareto set. E.g., when the functions at hand are convex, or relatively simple, or when the singularities are situated far away from interesting zones, a global investigation of the problem is not required. Moreover, typically, trade studies are performed in the neighborhood of a previously determined solution, therefore they can be limited to a non problematic branch of the Pareto set giving back as well the important information. 
Secondly, it is clear that a global exploration of the domain is a demanding task which could be far out of the scopes of a typical design problem. 

Nevertheless, faithful global representations of the Pareto set are a worthy goal to pursue, because they complement existing local exploitation strategies in two senses: they resolve the above mentioned problematic superpositions and 
they facilitate the location of important zones, which could merit further investigation. It is clear that this kind of program has to be implemented in an efficient way in order to be realistically useful in applications. On the other hand, even a roughly sketched global picture of the whole situation can give crucial information on the problem at hand, suggesting correctly the location of paramount zones. 


\subsection{Global analysis and multiobjective nonlinear programming}
With this in mind, we have devised a novel numerical strategy for approximating Pareto sets, theoretically based on the global analysis\footnote{See \cite{Smale:1969pt}. For brevity, we speak a bit loosely of global analysis also when referring to concepts of singularity theory or differential topology.} framework established by S. Smale and others in the early 1970's  
\cite{Melo:1976sj,Melo:1976xd,Debreu:1976ta,Debreu:1993if,
Smale:1973km,
Smale:1975oy,
Wan:1975bf,Wan:1975lt,
Wan:1978jy
} and in more recent work \cite{Miglierina:2004fr,Miglierina:2008gf}.
Motivated by his discussions with G. Debreu\footnote{ Debreu won the Nobel prize for Economics in 1983 \emph{``for having incorporated new analytical methods into economic theory and for his rigorous reformulation of the theory of general equilibrium''}. For an account of the cooperation between Smale and Debreu, see \cite{Debreu:1976ta,Debreu:1993if}.}, Smale investigated the problem of optimizing several functions within the dynamical systems arena.  In the series of works that followed there emerged interesting topological and geometrical features of the sets of the Pareto optima. The notion of \emph{Pareto critical set}  $\theta$, generalizing the concept of critical point for scalar functions,  was introduced and furthermore, local Pareto optima were characterized by means of first and second derivatives. Quoting Smale \cite{Smale:1973km}:
\begin{quote}\itshape
	``We study the local and global nature of $\theta$, as one uses freshman calculus to study the maximum of a single function.''
\end{quote}

One of the basic facts highlighted in Smale's global analysis framework, is that under the assumptions of second order differentiability and some generic transversality condition, Pareto optimal sets are portions of $(m-1)$--dimensional manifolds. It is fundamental that a slight deformation of the functions do not alter substantially the Pareto set. 
Global analysis is the proper setting where to study such kind of resilience properties. 
A mapping $u: W\to \R^m$ is said \emph{structurally stable}\footnote{
To be precise, we should speak of \emph{stability of mappings}, while \emph{structural} stability is more often used when speaking about differential equations. On the other hand we must speak about stability of Pareto optima, which is instead a concept deriving from the study of stability of equilibra, and refers to critical points which are maxima.  
Therefore we will keep speaking of \emph{structural} stability when dealing with typical singularities of mappings.} if there exists a neighborhood $N(u)$ in the $C^r$ topology such that every function $\tilde u$ in $N(u)$ is equivalent to $u$, i.e., there exist diffeomorphisms $h,k$, close to the identities of the respective spaces, such that the diagram:
$$
\xymatrix{
W \ar[r]^u  \ar[d]_h  & \R^m \ar[d]^k \\
W \ar[r]^{\tilde u}& \R^m
}
$$
commutes. 
Clearly if two mappings are equivalent, their Pareto sets are diffeomorphic. 
One of the main results of global analysis is that there exists an open and dense set in $C^r(W,\R^m)$ of structurally stable mappings.\footnote{It is necessary that $m<7$ and $n\neq 8$, or $m<6$ and $n=8$ \cite{Mather:1971sh}.} In other words, the Pareto set of \emph{almost} every mapping $u$ is as close as desired to the Pareto set of any other mapping in a sufficiently small neighborhood of $u$. This is clearly of fundamental importance for the applications: when functions are known only with a certain approximation, as usual in engineering design problems, 
 the set of optimal points is guaranteed to be approximated correctly by any convergent sequence of functions  \cite{Arnold:1968fx,Smale:1969pt,Thom:1972rw}.\footnote{The original idea of structural stability is a joint work from an engineer, A. Andronov, and a mathematician, L. Pontryagin, see \cite{Smale:1967p2934,Smale:1969pt}.}
Moreover, a generalization of Morse theory for several functions can be defined \cite{Smale:1973km,Wan:1975bf}.

The strategy presently  proposed highlights and exploits the manifold structure underlying the Pareto sets and precisely reproduces the hierarchy, described in Smale's work, among singular set, Pareto critical set and stable Pareto critical set.  
These sets are approximated by means of simplicial complexes, and exploiting Newton--type estimates it is possible to prove quadratic precision in set--wise sense, adopting the Hausdorff measure. Because of this result the present method can be considered as a set--wise variant of multiobjective Newton methods, as \cite{Fliege:2009kc}.

The algorithms of this paper can also be considered as a globalization and generalization to more than two objectives of the homotopy techniques, while the use of tessellations can be thought of as the specialization of the techniques of simplicial pivoting \cite{Allgower:2000p2573,Allgower:1991p2654,Allgower:1985p2832} to the problem of optimizing several functions. 
A strong similarity can be found in the method proposed by \cite{Schutze:2005fk}, where the authors detect and progressively refine the hypercubes containing the Pareto sets, relying on the standard Karush--Kuhn--Tucker conditions, instead of Pareto criticality.  

\section{The global analysis framework} 
We recall now Smale's definitions and results. 
Let $W\subseteq \R^d$ be an open set or more generally a smooth $n$--dimensional manifold, $u:W\to\R^m$ a smooth vector function, with $m\leq n$. 
\footnote{The case $m<n$ is less frequent. We will consider some aspect of this case in the sequel.}
The \emph{singular set} $\Sigma\subseteq W$, is the collection of \emph{singular points}, i.e., the points where the rank of the Jacobian $Du(x)$ is non maximal. If $m=1$, the singular set coincides with the set of critical, or stationary, points, i.e, $Du(x)=0$. It can be proved that under generic conditions the singular set is a smooth manifold. 
\\
Let $Pos$ be the open positive cone in $\R^m$, $Pos:= \set{y\in \R^m\taleche y_j>0, \forall j=1,\dots,m}$, and let $C_x$ the corresponding open cone in the tangent space $T_xW$, 
$C_x:=Du^{-1}(Pos)$.
\begin{defin}[Pareto critical set $\theta$] The set 
	\begin{equation}
		\theta := \set{x\in W\taleche C_x=\emptyset},
\end{equation}
is called the \emph{Pareto critical set}.
\end{defin}

We characterize $\theta$ in terms of the Jacobian of $u$.
\begin{prop}[First order proposition]\label{prop:firstorder} Let $x\in W$. Then, $x\in\theta$ if and only if:
\begin{enumerate}[(a)]
	\item $\set{Du_j(x)}_{j=1,\dots,m}$ do not belong to a unique open half space of the cotangent space $T^\star_x W$.
	\item $\exists \ \lambda_j \geq 0,\ j=1,\dots,m$, not all zero such that $\sum_j \lambda_j Du_j(x) = 0.$
\end{enumerate}
\end{prop}
\begin{rem}
	The meaning of Proposition \ref{prop:firstorder} is to exclude at the first order, for $x$ to be critical, the existence of paths along which all the objectives $u_j$ can be incremented at the same time. If there were an  open half space containing all $D u_j$, as in condition (a), any direction in this half space would be a direction of improvement for every $u_j$. Equivalently, condition (b) states that the gradients $Du_j$ should be linear dependent and furthermore should ``oppose'' each other. In other words, moving in the direction of maximal increasing according to one of the $u_j$  causes one or more of the remaining $u_i$ to strictly decrease. 
\end{rem}
\begin{rem}
 In the bi--objective case, $m=2$, Proposition \ref{prop:firstorder} states that in Pareto critical points the two gradients are  collinear and in opposition to each other. Also critical points for one of the two objectives are Pareto critical.
\end{rem}

In analogy with freshman calculus, (Pareto) criticality is a \emph{necessary} condition for $x$ to be optimal. 
In order to discriminate the nature of the Pareto critical points we introduce a notion of stability and point out its relation with the second derivatives of $u$. This will give \emph{sufficient} conditions for $x$ to be Pareto optimal.
\begin{defin}
A curve $(a,b)\owns t\mapsto \phi(t)\in W$ is said to be \emph{admissible} if
\begin{equation}
	\detot{}{t} u_i(\phi(t)) > 0, \qquad t\in (a,b), \quad \forall i=1,\dots,m.
\end{equation}
\end{defin}
Clearly, if a point is Pareto critical, there could not exist admissible curves passing through it. 
In order to establish its optimality it is necessary to investigate the behavior of the admissible curves in a neighborhood of a critical point. Admissible curves are smooth paths along which every objective is incremented. Therefore, they move towards local Pareto optima; conversely, if a critical point captures all neighboring admissible curves, that point is a local Pareto optimum. 
\begin{defin}
	A Pareto critical point x is said to be  \emph{stable}, $x\in\theta_S$, if, given a neighborhood $V_x$ of $x$ in $W$, there exists a neighborhood $U_x$ of $x$ in $V_x$, such that every admissible curve $\phi:[a,b) \to  W$, with $\phi(a)\in U_x$ satisfies $\mathrm{Image}(\phi) \subset V_x$.
\end{defin}
%
%
Pareto stability can be fully decidable by carefully examining  the second derivatives of the objectives. In the single objective case, by virtue of the Morse's Lemma, it is possible to find a coordinate system where the objective $u$ can be written as a quadratic polynomial $u=\pm x_1^2 \pm\dots\pm x_d^2$, which number of minus signs defines the Morse index, and therefore decides the nature of the critical point (maximum, minimum or saddle) \cite{Milnor:1963fk}. With some effort, results can be extended to multiple objectives: second derivatives are not defined invariantly, but if we think about them as a symmetric bilinear form restricted to the kernel of the differential $Du(x)$ assuming values on the cokernel $\R^m/ \mathrm{Image}(Du(x))$, then this form is invariantly defined. It is called \emph{``2nd intrinsic derivative''} (see \cite{Mather:1971sh,Porteous:1971kr}).
The restriction to the kernel of the tangent map $Du(x)$ has also the following meaning. By investigating the attractive/repulsive behavior of admissible curves in a neighborhood of a critical point, we will not be interested in what happens along the directions parallel to the critical set, while the orthogonal space will be the arena where the stability of the critical points will be decided.
The case of greatest importance is where corank $Du(x)$ is 1 (i.e., rank $Du(x)$ is $m-1$). In this case the second intrinsic derivative assumes values in a 1--dimensional vector space. If we consider $x\in\theta$, we have $\mathrm{Image}(Du(x)) \cap Pos =\emptyset$, thus $\R^m/\mathrm{Image}(Du(x))$ has a canonical positive ray. We call the 2nd intrinsic derivative, in this case, the \emph{generalized Hessian} $H_x$. It makes sense to say that $H_x$ is negative definite or positive definite, as well as to define an \emph{index}, as the index of the symmetric form $H_x$. We set 
\begin{equation}
	\partial \theta = \set{x\in\theta \taleche \image (Du(x)) \cap \set{ \mathit{Cl}({Pos}) \setminus \set{0}} \neq \emptyset}
\end{equation}
where $\mathit{Cl}({Pos})$ is the closure of $Pos$.
\begin{prop}[2nd order Proposition] \label{prop:secorder} Let $u:W\to\R^m$ a smooth map with $x\in\theta$, $x\not\in\partial \theta$ and corank $Du(x)=1$. Then 
	\begin{enumerate}[(a)]
		\item if the generalized Hessian $H_x$ is negative definite, then $x\in\theta_S$.
		\item Let $\lambda_j\geq 0$, $j=1,\dots,m$ be as in the 1st order proposition; then (up to a positive scalar) 
		\begin{equation}
			H_x = \sum_{j=1}^m \lambda_j D^2 u_j(x), \qquad \text{on } \ker Du(x).
		\end{equation}
	\end{enumerate}
\end{prop}
The proposition is proved in 
\cite{Smale:1975oy}, while a discussion of the genericity of the hypotheses on the rank assumption is given in \cite{Smale:1973km}. 

Most importantly, Proposition \ref{prop:secorder} offers a useful and workable criterion for deciding the stability of critical points. We will translate numerically this proposition in Algorithm \ref{alg:secorder}.

\subsection{The structure of Pareto sets} \label{ssec:structpareto}
We start by recalling the notion of Thom's stratification (see \cite{Thom:1964lc,Thom:1956nq,Thom:1977wq,Thom:1972rw}). 
\begin{defin} Let $A\subset W$ be a closed subset. A stratification $\mathcal{S}$ of $A$ is a finite collection of connected submanifolds of $W$ satisfying the following properties:
\begin{enumerate}[(1)]
	\item $\cup_{S\in\mathcal{S}} S = A$.
	\item If $S\in\mathcal{S}$ then $\partial S= \mathit{Cl}(S) \setminus S$ is a union of elements of $\mathcal{S}$ of lower dimension. 
	\item If $S\in\mathcal{S}$ and $U$ is a submanifold of $W$ transversal to $S$ at $x\in S$ then $U$ is transversal to all elements of $\mathcal{S}$ in a neighborhood of $x$.
\end{enumerate}
\end{defin}
The following theorem has been proved in \cite{Melo:1976sj}. Consider the space 
$C^\infty(W,\R^m)$  endowed with the $C^\infty$ topology. $W$ is a compact manifold with dimension $n\geq m$.

\begin{teo}[$\theta$ is a stratified set of dimension $m-1$] There is an open and dense set 
$\mathcal{G}\subset C^\infty\tonde{W,\R^m}$ such that if $u\in\mathcal{G}$ then $\theta$ is a stratified set of dimension $m-1$.
\end{teo}
\begin{rem} If $m>n$, it is possible to prove that, for a generic mapping $u$, $\theta$ is a stratified set of dimension $n$.
\end{rem}

From the point of view of the numerical applications, we state that in the generic case the strata of the Pareto critical set $\theta$ can be discretized  by means of a collection of $(m-1)$--dimensional meshes. Obviously we would like to refine this procedure to $\theta_S$. Unfortunately, the following conjecture has been proved only for $m=2,3$ (see\cite{Melo:1976xd,Wan:1975bf}).

\begin{conjecture}
	There is an open and dense set $\mathcal{G}\subset C^\infty(W,\R^m)$ such that if $u\in \mathcal{G}$ then $\theta$ is a stratified set and $\theta_S$ is a union of strata.
\end{conjecture}

\begin{rem} The stable Pareto critical set $\theta_S$ is formed by all the local Pareto optimal points. 
The global Pareto optimal points cannot be distinguished from local optima by means of differential features 
as in the statements presented above. 
Global Pareto optima can only be filtered out a posteriori. 
\end{rem}

\section{Numerical translation of the global analysis approach}
In the following sections we illustrate numerical methods for approximating Pareto sets on the basis of Propositions \ref{prop:firstorder} and \ref{prop:secorder}.  
The procedure is reminiscent of contour plot algorithms for plotting level sets of functions, and is a special instance of general strategies for piecewise--linear approximation algorithms for implicitly defined manifolds \cite{Allgower:1980p2815,Allgower:1985p2832,Allgower:2000p2573,Allgower:2002p2552}. The method determines a simplicial complex approximating the singular set $\Sigma$ and then refines it to the critical set $\theta$ and to the stable critical set $\theta_s$. Because the strategy proposed consists in a continuation method focussed on
the manifold structure of Pareto optima inherited by the singular set,  we coined the term \emph{singular continuation}.

\subsection{First order search algorithm} Algorithm \ref{alg:firstorder} translates numerically Proposition \ref{prop:firstorder}. We start by considering a set of data points $\mathcal{D}=\set{P_1,\dots,P_N}$ where we will evaluate the Jacobian $J_u$, then we build a Delaunay tessellation having $\mathcal{D}$ as nodes.\footnote{In the implementation considered  in what follows we employed the \texttt{qhull} software \cite{Barber:1996mk}, based on the computation of convex hulls, and, in the two dimensional examples, we employed the \textsc{triangle} software \cite{Shewchuk:1996xh,Shewchuk:2002wa}. For iterative schemes, a efficient alternative is offered by the Bowyer--Watson algorithm 
\cite{Bowyer:1981ud,Watson:1981kq}, which is incremental. } We assume that the nodes $P_1,\dots,P_n$ are in general position, i.e., they give rise to a valid Delaunay tessellation. Better results are obtained if the simplexes are ``round'', i.e., they do not possess very thin or very large angles. Special tessellations, e.g., Freuenthal--Kuhn, simplify the operation of ``pivoting'' from a simplex to the adjacent, speeding up the process of glueing together the polytopes composing the implicitly defined manifold\cite{Allgower:1985p2832}. Hereafter, we also assume that the dataset is sufficiently dense to resolve all the feature of the singular manifold $\Sigma$. More precisely, we assume that every connected component of $\Sigma$ intersects at least one of the $(n-m)$--faces $\Delta$ of the tessellation, and the intersection is unique and \emph{transversal}, i.e, $\dim T_x \Sigma \oplus T_x\Delta = n = \max$. Doing so $\Sigma$ is guaranteed to be homeomorphic to its piece--wise linear approximation. 
We denote by $\Sigma$, $\theta$ and $\theta_S$ the portions of the singular set, critical set and stable critical set, respectively, which possibly are contained in a simplex $\Delta$ of the tessellation of the domain considered. Hatted symbols, $\widehat\Sigma$, $\widehat\theta$ and $\widehat{\theta_S}$, denote the corresponding piecewise linear approximations. The details of the algorithm are discussed in subsection \ref{subsec:algo1desc}.
\begin{algorithm}
\caption{\it First order algorithm for approximating the Pareto critical set $\theta$}\label{alg:firstorder}
\begin{algorithmic}[1]
 \STATE Consider a set of data points  $\mathcal{D}=\set{P_1,\dots,P_N}$;
 \STATE evaluate the gradients of the $u_j$ on the data points;
 \STATE build a Delaunay tessellation on the nodes $\mathcal{D}$;
 \FORALL{ Delaunay simplex  $\Delta=\contrazione{P_{i_0},\dots,P_{i_n}}$ in the tessellation}
 	\STATE compute the $(m-1)$--polytope $\widehat\Sigma$ where the 1st order approximation of the Jacobian of $u$ vanishes;
	\STATE extract the sub--polytope $\widehat\theta$ where the vanishing linear combination $\lambda_1 Du_1 + \dots + \lambda_m Du_m = 0$ has non negative coefficients;
 \ENDFOR
 \STATE compose a simplicial complex glueing together adjacent polytopes $\widehat\theta$.
\end{algorithmic}
\end{algorithm}
\begin{rem}
The algorithm assumes $n\geq m$. 
When $m > n$ things extend quite easily, because the singular set is all of the input domain, and as recalled in Section  \ref{ssec:structpareto} the critical set is a stratified set. More precisely, the gradients are always linearly dependent, thus it is sufficient to skip step 5 of Algorithm \ref{alg:firstorder}.
\end{rem}
\subsection{Analysis of simplexes}\label{subsec:algo1desc}
We cycle through the tessellation simplexes $\Delta  = \contrazione{P_{i_1},\dots, P_{i_{n+1}}}$ and approximate the portion of the Pareto critical set $\theta$ possibly contained in $\Delta $.
To determine the linear approximation $\widehat\theta_s$  of the stable Pareto critical portion $\theta_s\cap \Delta $  we recall 
that $\theta$ is contained in the \emph{singular set } $\Sigma$, i.e., the 
set where the rank of the differential $Du(x)$ is  less than maximal:
\begin{equation}
 \theta_s \subseteq \theta \subseteq \Sigma \subseteq W, \qquad
(\Rightarrow \quad \widehat\theta_s \subseteq \widehat\theta \subseteq \widehat\Sigma \subseteq \Delta. ) 
\end{equation}
Adjacent approximate portions $\widehat \theta_s$ are eventually sewed together. 


\subsubsection{Singular set $\widehat\Sigma$} 
We fix a cell $\Delta :=\contrazione{P_1,\dots,P_{n+1}}$.
The Jacobian is an $n\times m$ matrix which rank is non maximal on the singular set $\Sigma$. The rank of $J_u$ drops when the rows are linearly dependent, e.g., when a suitable subset of the $m$--order minors are degenerate. We consider for instance the following submatrices:\footnote {Apart from degenerate cases, all $m$--minors share the same rank, so it is sufficient to consider only a selection of minors involving at least once each of the columns of the Jacobian.} 
\begin{multline}
	M_1 = 
	\begin{pmatrix}
		\deinde{u_1}{x_1} & \dots & \deinde{u_1}{x_m} \\ \vdots & \ddots & \vdots  \\
		\deinde{u_m}{x_1} & \dots & \deinde{u_m}{x_m}
	\end{pmatrix}, \qquad 
	M_2 = 
	\begin{pmatrix}
		\deinde{u_1}{x_2} & \dots & \deinde{u_1}{x_{m+1}} \\ \vdots & \ddots & \vdots  \\
		\deinde{u_m}{x_2} & \dots & \deinde{u_m}{x_{m+1}}
	\end{pmatrix}, \dots\\
	\dots, \qquad
	M_{n-m+1} = 
	\begin{pmatrix}
		\deinde{u_1}{x_{n-m+1}} & \dots & \deinde{u_1}{x_n} \\ \vdots & \ddots & \vdots  \\
		\deinde{u_m}{x_ {n-m+1}} & \dots & \deinde{u_m}{x_n}
	\end{pmatrix}. 
\end{multline}
We denote the number of minors by $r:=n-m+1$ and set $\upomega_j(x) := \det M_j(x)$, for $j=1,\dots,r$ and consider all the $(r+1)$--faces of the cell $\Delta$, i.e.,  for  every $\set{i_1,\dots,i_{r+1}}\subseteq\set{1,\dots,n+1}$, with $i_1<i_2<\dots<i_{r+1}$, 
we consider the simplex\\ 
$\contrazione{P_{i_1},\dots,P_{i_{r+1}}}$. 
The solution $(\mu_{1},\dots,\mu_{r+1})$ of the system:
\begin{equation}
	\begin{cases}
		\mu_{1} \upomega_1(P_{i_1}) + \dots + \mu_{r+1} \upomega_1(P_{i_{r+1}}) & = 0\\
		\vdots \qquad \vdots \qquad \vdots & \vdots \\
		\mu_{1} \upomega_{r+1}(P_{i_1}) + \dots +\mu_{r+1} \upomega_{r+1}(P_{i_{r+1}})& = 0\\
		\mu_{1} +\dots +\mu_{r+1} & = 1		
	\end{cases}\label{eq:minorsystem}
\end{equation} 
leads to a \emph{singular vertex} $Q:= \mu_1 P_{i_1} + \dots + \mu_{r+1} P_{i_{r+1}} $ of $\widehat\Sigma$ if all $\mu_{j} >0$, i.e., if $Q$ is contained in the $(r+1)$--face of $\Delta$ considered. 

The (possibly empty) singular set $\widehat\Sigma$ is an $(m-1)$--polytope defined as the convex hull of the singular vertices $Q$. 

\subsubsection{Critical set $\widehat\theta$}

%

In the previous subsection we have detected the singular set $\Sigma$, on the basis of the fact that on the singular set the gradients are linearly dependent. On the other hand, on the critical set $\theta$ there exists a positive linear combination of the gradients giving zero. Thus we proceed by estimating 
the coefficients $\lambda_j$ of the vanishing  linear convex combination of the gradients, and cutting out the critical set $\theta$ from $\Sigma$ by intersection with the half spaces where the linear interpolations of the $\lambda$s are positive. 

More precisely, we solve the system:
\begin{equation}\label{eq:lambdasys}
\begin{cases}
	\lambda_1 Du_1(P)+\dots+\lambda_m Du_m(P) = 0, \\
	\lambda_1 +\dots+\lambda_m  = 1, 
\end{cases}
\end{equation}
for $\lambda_1,\dots,\lambda_m$.
The Jacobian of $u$ has rank $m-1$ in almost all the points of the singular set (generic hypothesis), thus the system (\ref{eq:lambdasys}) has rank $m$, and by the implicit function theorem $\lambda_j$ are smooth functions of $P$. As a result the level sets $\set{\lambda_j(P)=0}$, which define the boundary of $\theta$, are smooth manifolds. At the first order we are working with, the requests $\lambda_j(P)\geq 0$ cut out half spaces in $\widehat \Sigma$, defining  possibly a critical sub polytope $\widehat \theta$ in $\Delta$.

We notice that we do not know the actual values of $Du$ on the singular vertices, i.e., the nodes of $\widehat\Sigma$. 
Nevertheless, we can estimate them by linearly interpolating the values of $Du$ on the data nodes $P_{i_1},\dots, P_{i_{r+1}}$ defining the vertex $Q$ in $\widehat\Sigma$. By taking the coefficients $\mu_{1},\dots,\mu_{r+1}$ solving the system (\ref{eq:minorsystem}), we are guaranteed that the $\widehat{Du}_j(Q) := \mu_{1} Du_j(P_{i_1}) +\dots+\mu_{r+1} Du_j(P_{i_{r+1}})$ are linearly dependent, and we are justified in solving for the vanishing linear combination $\lambda_1 \widehat {Du}_1(Q)+\dots +\lambda_m \widehat {Du}_m(Q)=0$.

\subsection{Convergence analysis for $\theta$}\label{sssec:convergence} Let us consider for this section a single simplex $\Delta$. 
Intuitively, it is clear that the approximation $\widehat\Sigma$ of $\Sigma$ obtained by linear interpolation is quadratically good because of Taylor's theorem. 
We state more precisely this result in the set--wise context we have adopted.\footnote{General estimates on the accuracy of piecewise--linear approximations of implicitly defined manifolds are proved in \cite{Allgower:1989p2804}.}
The distance between the sets $A$ and $B$ can be measured in terms of Hausdorff distance:
\begin{equation}
	d_\mathcal{H}(A,B) :=  \max \tonde{ \sup_{x\in A} \inf_{y\in B} d(x,y), \sup_{y\in B} \inf_{x\in A} d(x,y)}.
\end{equation}
\begin{teo}[Quadratic precision for $\Sigma$] \label{teo:quadraticsing} Let $P_0,\dots,P_{n}$ be in general position, and such that $Du$ has maximum rank. We denote by $\Delta =\contrazione{P_0,\dots,P_{n}}$ the $n$--simplex which vertices are those points. Let $\upomega_1(x),\dots,\upomega_{r}(x)$ a selection of independent minors of $Du$, and $\widehat \upomega_j(x)$
be the 1st order interpolation of the values of $\upomega_j$ on the nodes $P_i$. Assume $0$ is a regular value for $\upomega_1,\dots,\upomega_r$, that the zero levels of the $\upomega_j$ are transversal and that $\upomega_j(P_i) \neq 0$, for all $i,j$. 
Then
\begin{gather}
	\Sigma = \set{\upomega_1(x)=0} \cap \dots \cap \set{\upomega_{r}(x)=0}, \\ 
	\widehat\Sigma = \set{\widehat \upomega_1(x)=0} \cap \dots \cap \set{\widehat \upomega_{r}(x)=0},
\end{gather}
and there exists a constant $\mathsf{C}$:
\begin{equation}
	d_\mathcal{H}(\Sigma,\widehat \Sigma) \leq \mathsf{C} \delta^2,
\end{equation}
where $\delta>0$ is the diameter of the simplex $\Delta$.
\end{teo}
\begin{proof} First of all, we notice that the $\upomega_k(x)$ are polynomials of the first derivatives of $u$, thus they are smooth in our hypotheses.	Inductively, consider $r=1$ and denote $\upomega=\upomega_1$. By Taylor's theorem, 
	\begin{equation}
		\begin{split}
 	\upomega(x) = \widehat{\upomega}(x) + O\tonde{\abs{x-P_0}^2}, \qquad \text{i.e.,}\\
	        		\abs{\upomega(x)-\widehat{\upomega}(x)} \leq C\delta^2,
		\end{split}
	\end{equation}
for a suitable $C>0$. Assume, without loss of generality, $\upomega>0$ on $P_0,\dots,P_k$ and $\upomega<0$ on $P_{k+1},\dots,P_N$. Let $\ep:= C\delta^2$. (See panel (a) of Figure \ref{fig:convproof}).
\begin{figure}[htbp]
\begin{center}
\begin{tabular}{ccc}
 \includegraphics[width=40mm]{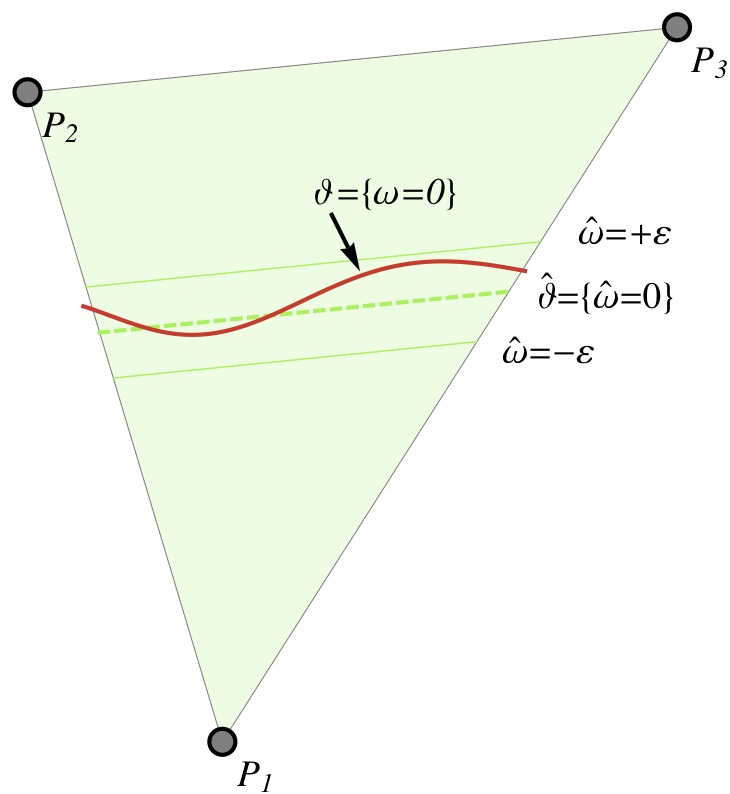} &
	\includegraphics[width=38mm]{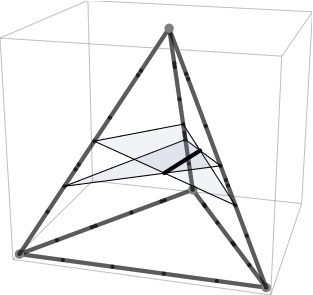} &
	\includegraphics[width=38mm]{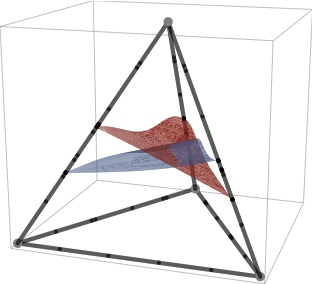}\\
	(a) & (b) & (c)
\end{tabular}
\caption{Critical simplexes, with representations of the critical set $\theta$ and its first order approx $\widehat\theta$. Panel (a): two functions in two dimensions. Panel (b) and (c): two functions in three dimensions. In panel (b) the the thick line is the first order approximation $\widehat\theta$, in panel (c) the critical set $\theta$ is the curve of intersection of the two level surfaces $\upomega_1(x)=0$ and $\upomega_2(x)=0$.  }
\label{fig:convproof}
\end{center}
\end{figure}
Thus the zero levels of $\upomega$ and $\widehat \upomega$ are comprised between the $\pm\ep$ levels of $\widehat{\upomega}$, i.e., 
\begin{equation}
	\set{x\in \Delta\taleche \upomega(x) = 0} \subseteq \set{x\in \Delta\taleche -\ep \leq \widehat \upomega(x) \leq \ep}.
\end{equation}
By the compactness of $\Delta $, there exist $x_0\in \set{\widehat{\upomega}=0}$, $x_\ep\in \set{\widehat{\upomega}=\ep}$, such that, 
\begin{equation}
	d_{\mathcal{H}}\tonde{\set{\widehat{\upomega}=0},\set{\widehat{\upomega}=\ep}} = \abs{x_0 - x_\ep}
\end{equation}
and it holds that
\begin{equation}
	\widehat{\upomega}(x_\ep) - \widehat{\upomega}(x_0) =\abs{ \widehat{\upomega}'(x_0)\cdot (x_\ep-x_0) } = \abs{\deinde{\widehat{\upomega}}{w}(x_0)} \abs{x_\ep-x_0},
\end{equation}
where $w=\frac{x_\ep-x_0}{\abs {x_\ep-x_0}}$. By means of an elementary linear algebra argument we have also that 
\begin{equation}
	\abs{\deinde{\upomega}{w}(x_0)} \geq \min_{\substack{i=1,\dots,k,\\ i'=k+1,\dots,n}} 
		\abs{\frac{\widehat{\upomega}(P_i) - \widehat{\upomega}(P_{i'})} {P_i - P_{i'}}} =: B > 0, 
\end{equation}
so we can conclude 
\begin{equation}
	\abs{x_0 - x_\ep} \leq \frac{\ep}{B} = \frac{C}{B}\delta^2 = \mathsf{C} \delta^2
\end{equation}
and eventually
\begin{equation}
	d_{\mathcal{H}}\tonde{\set{\upomega=0},\set{\widehat \upomega = 0}} \leq \mathsf{C} \delta^2.
\end{equation}

Consider now $r>1$, and assume inductively that the Hausdorff distance between the intersection of the zero levels of $r-1$ transversal functions and the intersection of the zero level of the respective linear interpolations on an $n$--simplex is quadratically smaller than the simplex diameter. Thus we have 
\begin{gather}
	\Sigma_- = \set{\upomega_1(x)=0}\cap\dots \cap \set{\upomega_{r-1}(x)=0},\\
	\widehat\Sigma_- = \set{\widehat  \upomega_1(x)=0}\cap\dots \cap \set{\widehat  \upomega_{r-1}(x)=0}, \\
	d_\mathcal{H}\tonde{\Sigma_-, \widehat\Sigma_- } \leq C \delta ^2.
\end{gather}
If we consider one more function $\upomega_r(x)$ on the linear space $\widehat\Sigma_-$, we are in the previous case, so there exists $A>0$,
\begin{equation}
	d_{\mathcal{H}}\tonde{\widehat \Sigma_- \cap \set{\upomega_r(x)=0},  \widehat\Sigma_- \cap \set	
	{\widehat \upomega_r(x)=0}} \leq A \delta^2.
\end{equation}
By transversality of the $\upomega_1,\dots,\upomega_r$, the fact holding for the linear space $\widehat\Sigma_-$ holds also for the compact manifold with boundary $\Sigma_-$ and the function $\upomega_r$ (see Lemma \ref{lem:mani} for the details).  Thus there exists a $B>0$ such that 
\begin{equation}
	d_{\mathcal{H}}\tonde{\Sigma_- \cap \set{\upomega_r(x)=0}, \Sigma_- \cap \set	
	{\widehat \upomega_r(x)=0}} \leq B \delta^2.
\end{equation}
On the other hand, for the intersection of the zero levels of the transversal functions $\upomega_1$, \dots,$\upomega_{r-1}$
 on the linear space $\set{\widehat \upomega_r(x)=0}$, by the inductive hypothesis there exists $C>0$
\begin{equation}
	d_{\mathcal{H}}\tonde{\set{\upomega_r(x)=0}\cap \Sigma_-,  \set{\upomega_r(x)=0}\cap\widehat \Sigma_- } \leq C \delta^2,
\end{equation}
 so the thesis is proved by the triangle inequality.
\end{proof}

\begin{lem}\label{lem:mani} Let $\Sigma$ a manifold with boundary diffeomorphic to an $n$--simplex $\Delta$, and $\upomega:\Sigma\to\R$ differentiable and without critical points inside $\Sigma$. 
We have $\upomega(x)=\widehat \upomega(x) +O(\delta^2)$, where $\widehat \upomega$ is an affine approximation and $\delta$ is the simplex diameter. Thus we have that 
\begin{equation}
	d_\mathcal{H}\tonde{\set{\upomega(x)=c},\set{\widehat \upomega(x) = c}} \leq C \delta^2, \qquad \text{for all} \ c\in \R.
\end{equation}
\end{lem}
\begin{proof}
	Let $\Delta\overset{\phi}{\To} \Sigma$ be a diffeomorphism, with $\xi> \abs{\phi'} > \eta > 0$. Thus we have, for all $y\in \Delta $, 
	$$\upomega\circ\phi(y) = \widehat \upomega\circ\phi(y) + O(\delta^2).$$
For any $y^\star$ in the zero level of $\upomega\circ\phi$ we can find a line segment $[y_1,y_2]$, with $y_1$ being one of the nodes of $\Delta $ where $\upomega\circ\phi$ is negative and $y_2$ is a point on a face of  $\Delta$ where on the forming nodes $\upomega\circ\phi$ is positive. 
By continuity there exists a point $\widehat y$ on the line $[y_1,y_2]$ where $\widehat \upomega\circ\phi$ is zero. 

Thus
\begin{equation}
	\upomega\circ\phi(y^\star) - \upomega\circ\phi(\widehat y) = \widehat \upomega\circ\phi(y^\star)  - \widehat \upomega\circ\phi (\widehat y) + O( \delta^2) = \widehat \upomega' \circ \deinde{\phi}{w} \abs{y^\star-\widehat y} + O(\delta^2), 
\end{equation}
which gives 
\begin{equation}
	\abs{y^\star-\widehat y} \leq \mathsf{C} \delta^2.
\end{equation}
\end{proof}
\begin{osserv} The hypotheses of Theorem \ref{teo:quadraticsing} are \emph{generic} in the sense that they hold for a open and dense set of functions. In particular, 0 is assumed to be a regular value for $\upomega_1,\dots,\upomega_r$ because the set of the singular values has zero measure (Sard's Theorem). See \cite{Arnold:1968fx,Guillemin:1974hs,Levine:1976bs,Mather:1971sh,Michor:1985wb}.
\end{osserv}

\begin{teo}
	In the simplex $\Delta =\contrazione{P_0,\dots,P_n}$, if $\theta$ is the Pareto critical set and $\widehat\theta$ is its linear approximation, there exists $ \mathsf{C}>0$ such that
	\begin{equation}
		d_{\mathcal{H}} \tonde{\theta,\widehat \theta} \leq \mathsf{C} \delta^2.
	\end{equation}
\end{teo}
\begin{proof}
	The $\lambda_j$ computed as described in Algorithm \ref{alg:firstorder} are first order approximations to smooth functions, apart from a measure zero set of points. Thus the conclusions of Theorem \ref{teo:quadraticsing} 
	apply as well to the intersection of $\Sigma$ with the half spaces $\lambda_j(P) \geq 0 $. 
\end{proof}


\subsection{Second order algorithm} In Algorithm \ref{alg:secorder} we describe how to extract the stable critical set $\theta_s$, i.e., the set of locally Pareto optimal points, from the critical set $\theta$ determined in the first order algorithm \ref{alg:firstorder}. 
\begin{algorithm}[H]
 \caption {\it Second order algorithm for the stable Pareto critical set $\theta_s$}
 \label{alg:secorder}
\begin{algorithmic}[1]
 \STATE Consider a set of data points  $\mathcal{D}=\set{P_1,\dots,P_N}$ and proceed as in Algorithm \ref{alg:firstorder}.
 \FORALL{ Delaunay simplex  $\Delta=\contrazione{P_{i_0},\dots,P_{i_n}}$ in the tessellation}
 	\STATE Compute the matrix  of the second derivatives $D^2u$  on the nodes $P_{i_0},\dots,P_{i_n}$	
	\STATE On the vertices $Q$ of $\widehat\theta$, linearly interpolate the second derivatives $\widehat {D^2u}(Q)$
	\STATE Compute a basis $w_1,\dots,w_{n-m+1}$  for $\ker Du(Q)$, and set 
	$\widehat H (Q) := w^\top\cdot \tonde{\lambda_1(Q) \widehat {D^2u_1}(Q)+\dots+ \lambda_m(Q) \widehat{ D^2u_m}(Q)} \cdot w$. 
	\STATE Compute the eigenvalues $\sigma_1,\dots,\sigma_{n-m+1}$ of $\widehat H(Q)$.
	\STATE Cut out from $\widehat\theta$ the sub polytope $\widehat\theta_s$ where $\sigma_k\leq 0$ for all $k=1,\dots,n-m+1$
 \ENDFOR
 \STATE Compose a simplicial complex glueing together adjacent polytopes $\widehat\theta_s$
\end{algorithmic}
\end{algorithm}
The second derivatives could be also approximated computing the finite differences of the values of the gradients on the nodes of the $n$--simplex.
Indeed, setting 
$$
	v_i = P_i - P_0, \qquad i=1,\dots,n,
$$
we have 
\begin{equation}
	D^2u = \tonde{\deindesecdue{u}{x_i}{x_j}}_{i,j} = \sum_k \deindesecdue{u}{v_k}{x_j}\cdot\deinde{v_k}{x_i} \simeq \sum_k \tonde{\nabla u(P_k) - \nabla u(P_0)}_j\cdot \tonde{P_k-P_0}_i.
\end{equation}

%

Using this formula,  the quadratic precision cannot be guaranteed for locating boundary points of the stable critical set. Furthermore, because the boundary faces  belong to different simplexes, 
the estimated boundary points for $\theta_s$ would jump from  simplex to simplex. 

On the other hand, the formula will be correct for discriminating the nature of inner stable critical points,  
without extra computations. Boundary simplexes can thus analyzed with second derivatives, allowing the computation of the boundary of $\widehat\theta_s$.
 

\section{Applications}
\subsection{Two functions in two dimensional examples}

A series of examples in two dimensions is presented below. 
Via global analysis one sees that, for structurally stable mappings, the Pareto critical set is a 1--dimensional manifold with boundary contained in $\Sigma$.
Critical points can only be of one of the following types:
\begin{enumerate}
 \item \emph{fold}, i.e., the mapping is locally equivalent to $u_1 = x_1$, $u_2 = x_2^2$,
 \item \emph{cusp},  i.e.,   the mapping is locally equivalent to  $u_1 = x_1$, $u_2 = x_1x_2 - \frac{1}{3}x_2^3$.
\end{enumerate}
Therefore, the branches of Pareto critical points are composed by folds, which intersect only pair--wise and at non--zero angles. Some local branches terminate in cusps, where the status of critical points can change from stable to unstable.
Finally, images of folds and cusps do not intersect \cite{Arnold:1968fx,Wan:1975bf}.

Functions gradients are evaluated on a grid of regular triangles and the critical set $\theta$ is estimated according to the first order algorithm. 
Boundary points are marked with black diamonds.
The generalized Hessian is estimated on the nodes of the critical set, computing second derivatives in the triangles where its index changes, allowing to estimate the position of the  points separating stable from unstable branches. 
Cusps are marked by a black star, stable branches are colored in red, unstable branches in orange and finally non critical branches are gray.
\begin{esem} \label{ex:triv}
Consider two negative definite quadratic polinomials. The critical stable set is a curve joining the two individual critical points. Other singular branches occur in outer regions of the domain.
\begin{equation}
   \begin{split}
	u_1(x,y) & =  -1.05 x^2 - 0.98 y^2,\\
	u_2(x,y) & =-0.99(x - 3)^2 - 1.03(y - 2.5)^2.
   \end{split}
\end{equation}
See Figure \ref{fig:triv}.
\begin{figure}
\centering
\begin{tabular}{cc}
\includegraphics[width=75mm]{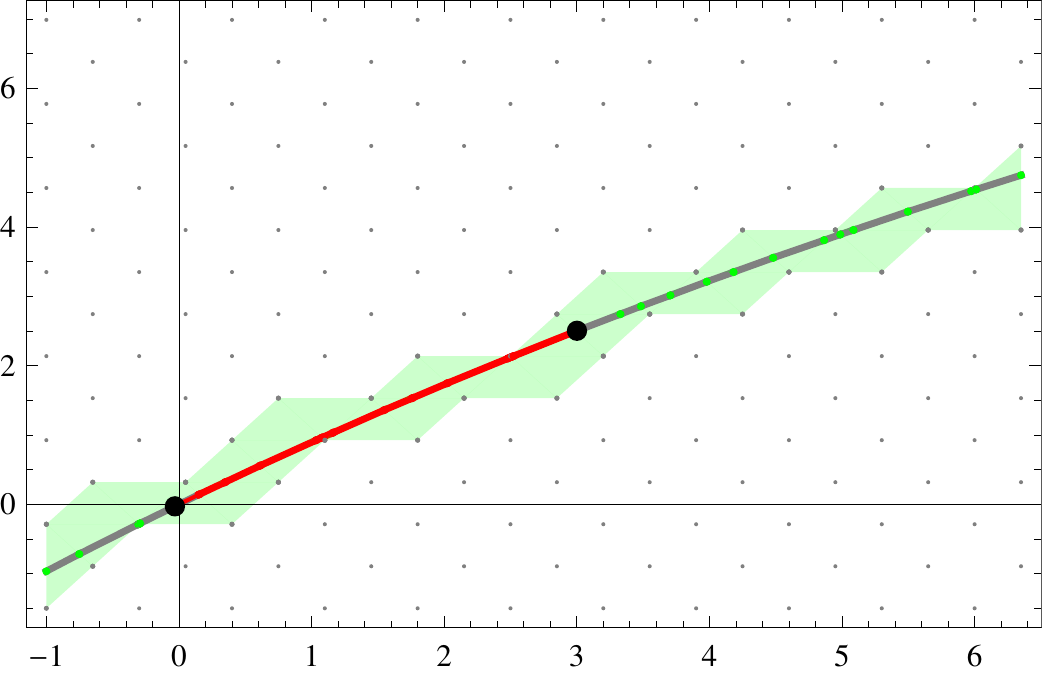}&
\includegraphics[width=45mm]{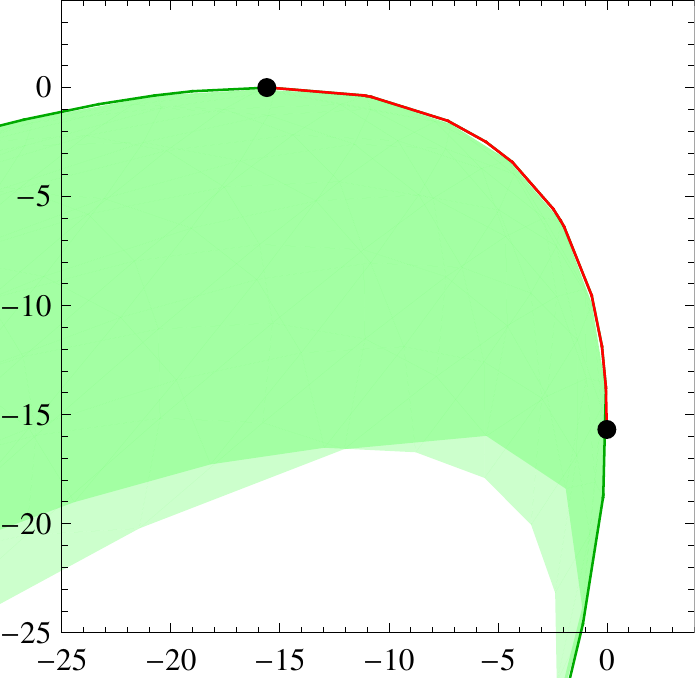} \\
(a) & (b)
\end{tabular}
\caption{Pareto critical set (a) and the Pareto front (b), for Example \ref{ex:triv}.}
\label{fig:triv}
\end{figure}
\end{esem}

\begin{esem} \label{ex:sma1}
This example is taken from  \cite{Smale:1975oy}.\begin{equation}
   \begin{split}
	u_1(x,y) & =  -y,\\
	u_2(x,y) & =  \frac{y - x^3}{x + 1}.
   \end{split}
\end{equation}	The critical set is a single curve split in a stable and an unstable branch, while the separating point is a cusp.
See Figure \ref{fig:sma1}.
\begin{figure}
\centering
\begin{tabular}{cc}
\includegraphics[width=75mm]{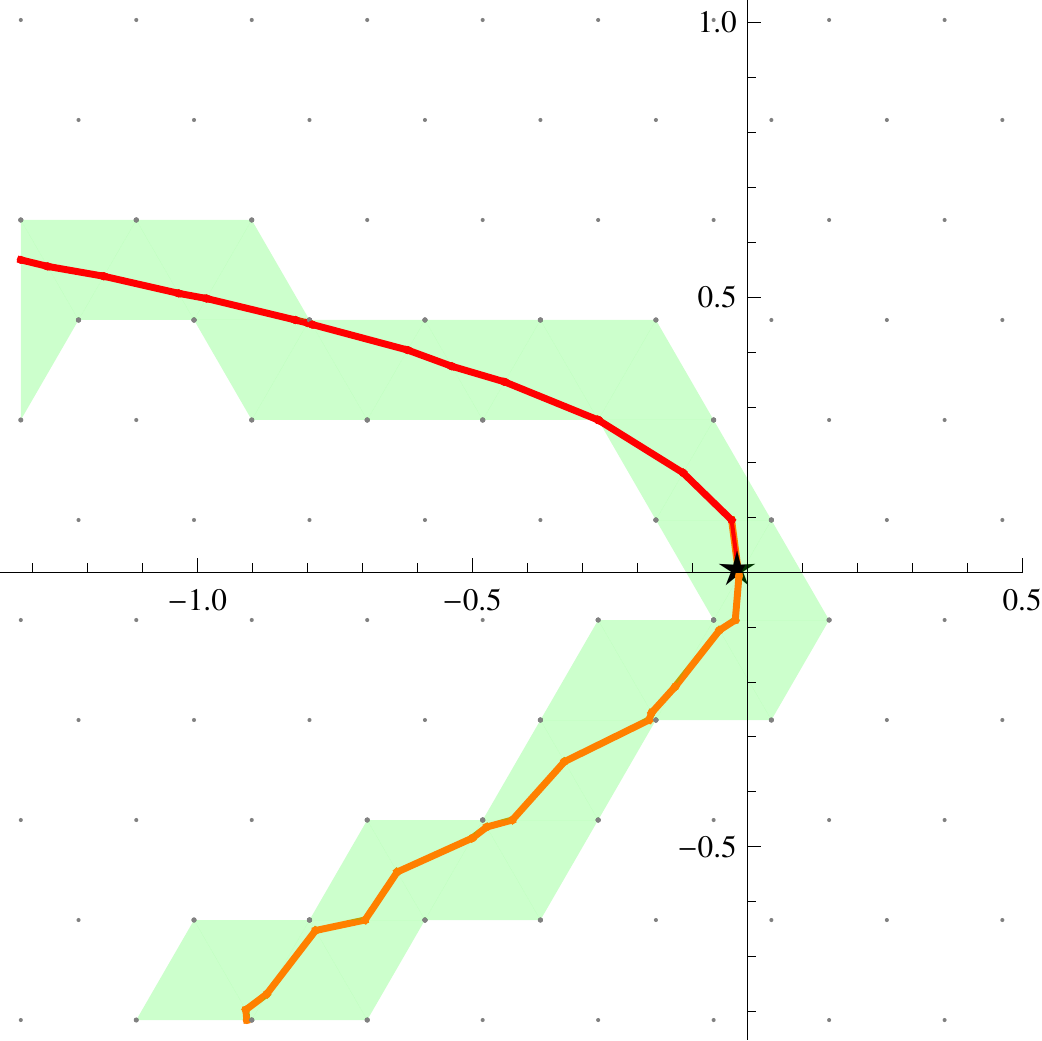}&
\includegraphics[width=45mm]{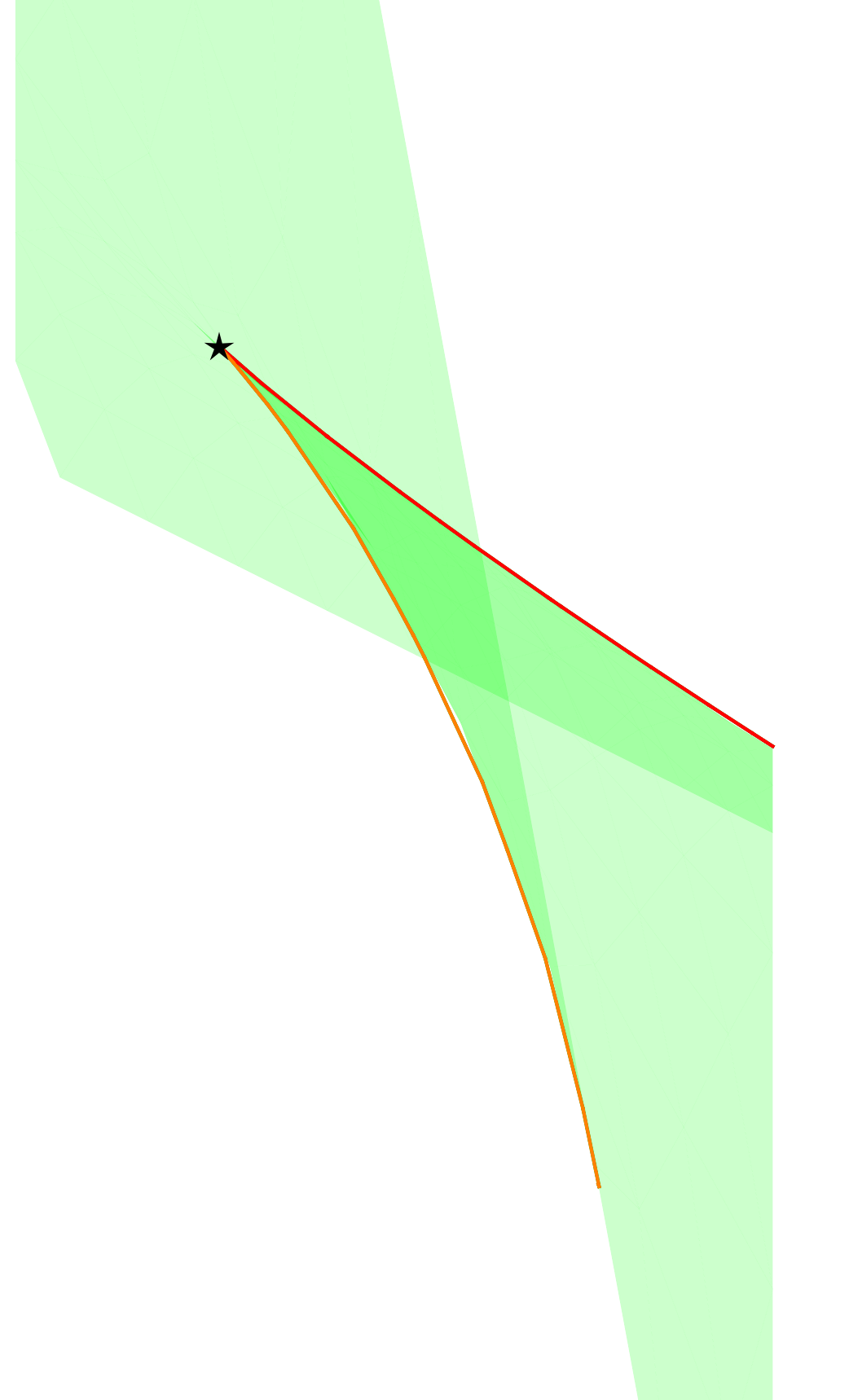} \\
(a) & (b)
\end{tabular}
\caption{ Example \ref{ex:sma1}. Red line: stable critical set. Orange line: unstable critical set.}
\label{fig:sma1}
\end{figure}
\end{esem}

\begin{esem} \label{ex:sms}
In the following mapping there are two second order polynomials, one negative definite and the other indefinite. The outcome is an (unbounded) global Pareto front and a local unbounded front terminating in a cusp.
\begin{equation}
   \begin{split}
	u_1(x,y) & =  -x^2 - y^2,\\
	u_2(x,y) & =-(x - 6)^2 + (y + 0.3)^2,
   \end{split}
\end{equation}	
See Figure \ref{fig:sms}. 
\begin{figure}
\centering
\begin{tabular}{cc}
\includegraphics[width=75mm]{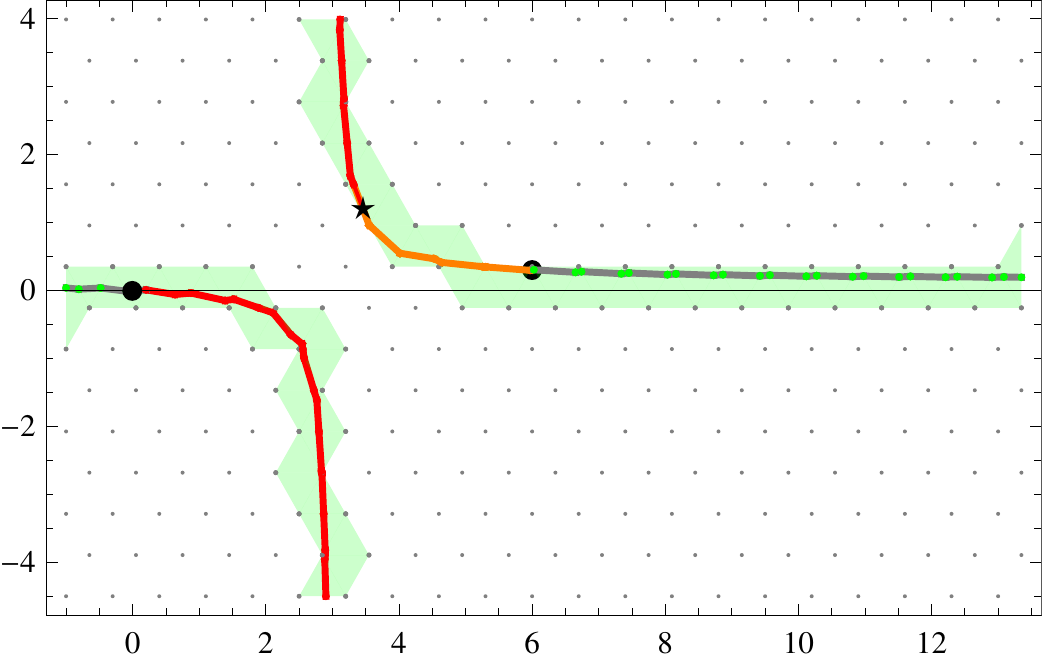}&
\includegraphics[width=45mm]{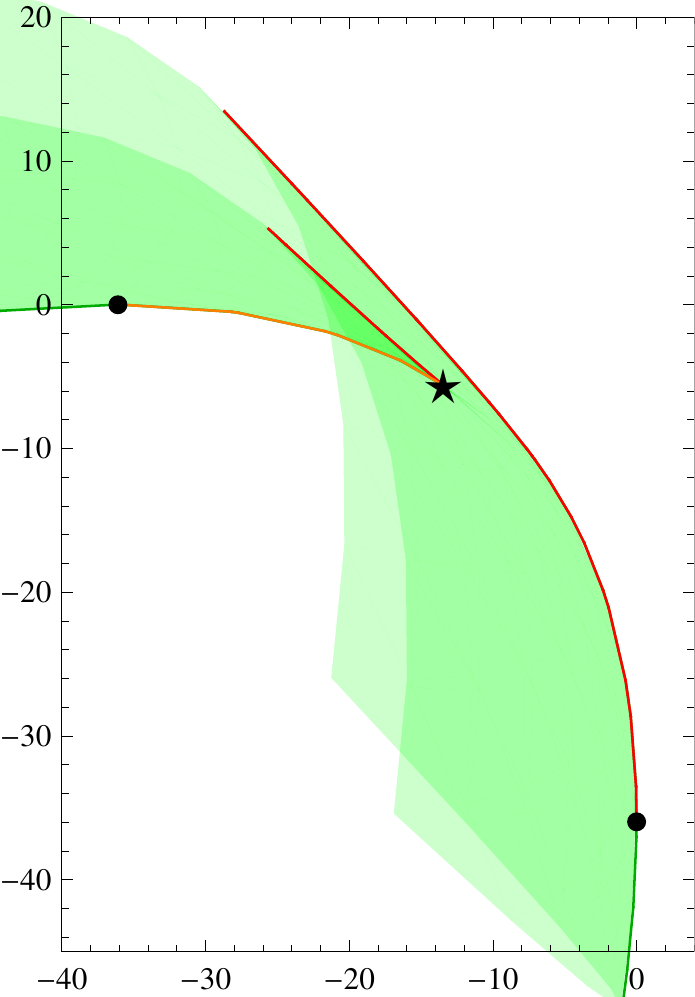} \\
(a) & (b)
\end{tabular}
\caption{Example \ref{ex:sms}.}
\label{fig:sms}
\end{figure}
\end{esem}

\begin{esem} \label{ex:noncv}
The following mapping is composed by a quadratic polynomial and a bimodal function. The resulting singular set is composed by an unbounded branch and two loops. One of the loops is critical and forms a local Pareto front delimited by two cusps, while the other loop is non critical. 
\begin{equation}
   \begin{split}
	u_1(x,y) & =  -x^2 - y^2 -   4 (\exp(-(x + 2)^2 - y^2) +  \exp( - (x - 2)^2 - y^2)),\\
	u_2(x,y) & =-(x - 6)^2 -(y + 0.5)^2.
   \end{split}
\end{equation}	
See Figure \ref{fig:noncv}
\begin{figure}
\centering
\begin{tabular}{cc}
\includegraphics[width=75mm]{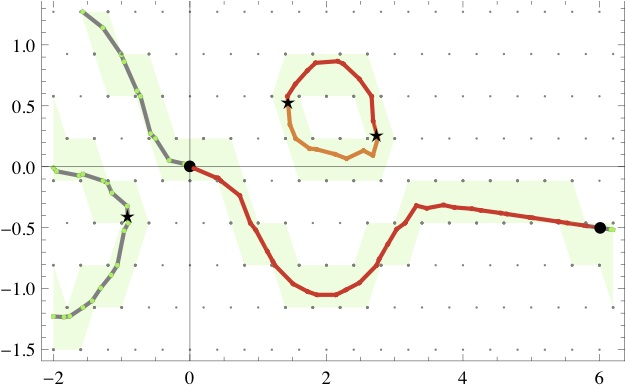}&
\includegraphics[width=45mm]{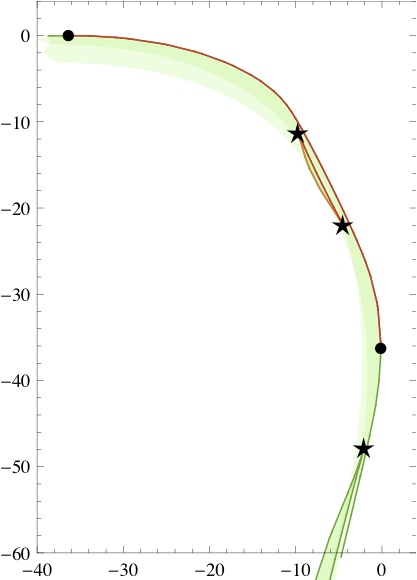} \\
(a) & (b)
\end{tabular}
\caption{Example \ref{ex:noncv}.}
\label{fig:noncv}
\end{figure}
\end{esem}

\subsection{Higher input dimension}

\begin{esem} \label{ex:locglob}
The following mapping demonstrates the capabilities of the method in distinguishing local and global features of the Pareto set. A widespread optimal branch is surpassed by a local branch. The sharper branch is composed by an unstable part (orange) and a stable part (red) which  is interrupted by noncritical insertions (gray). 
Nevertheless, as illustrated in figures \ref{fig:locglob}(a--b) the two separate branches are properly recognized by the algorithm and moreover the transitions among critical/non critical and stable/unstable intervals are detected. For comparison, the outcome of the application of a commercial implementation of normal boundary intersection by Das and Dennis  \cite{Das:1998zh} is shown in figures \ref{fig:locglob}(c--d).\footnote{Applications of \textsf{modeFRONTIER}$^\circledR$ are a courtesy by E. Rigoni at \textsc{esteco}}.
The starting grid (green dots) was $10\times20\times10$, and we considered 50 NBI subproblems. The sequence of NBI points is marked by black stars.   
For this particular problem, NBI tracks correctly the broad Pareto optimal branch, in the sense that it produces a parametrization of it.  However, the smaller branch is missed, although some of the points of the starting grid were close to this critical zone. 
It is clear that point--wise strategies suffer at tracking the Pareto optimal set and fail at performing widespread trade studies, apart from small intervals where different fronts are far apart and do not change status (from critical to non--critical, stable to unstable, and so on).  
\begin{equation}
   \begin{split}
p_0 &= (0.0, 0.15, 0.0)^\top,\\
p_1 &= (0.0, -1.1, 0.0)^\top,\\
M &= 
\begin{pmatrix}
 -1.0 & -0.03 & 0.011\\
 -0.03 & -1. & 0.07\\ 0.011 & 0.07 & -1.01
\end{pmatrix}, \\
g(x,y,z,M,p,\sigma)&= \sqrt{\frac{2\pi}{\sigma}} e^{\frac{{((x,y,z)^\top-p)^\top M ((x,y,z)^\top-p))}}{\sigma^2}}, \\
f(x,y,z)& =  
 g(x,y,z,M,p_0,0.35) +g(x,y, 0.5 z,M,p_1,3.0), \\
u1(x,y,z)& = \frac{\sqrt{2}}{2} x + \frac{\sqrt{2}}{2} f(x,y,z), \\
u2(x,y,z)& = - \frac{\sqrt{2}}{2} x + \frac{\sqrt{2}}{2} f(x,y,z). 
   \end{split}
\end{equation}	
%
\begin{figure}
\centering
\begin{tabular}{cc}
\hspace*{-5mm} \includegraphics[width=78mm]{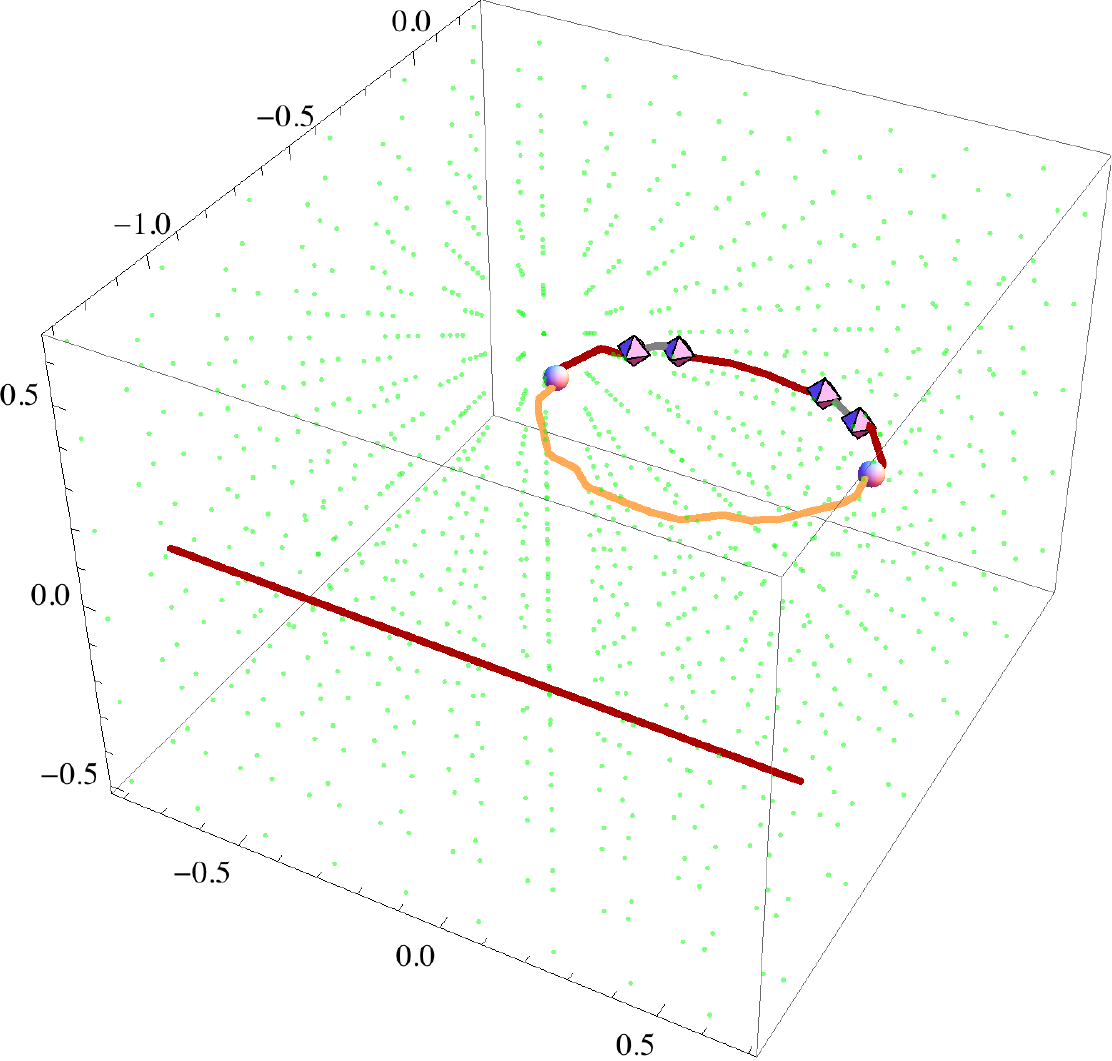} & \hspace*{-7mm}
\includegraphics[width=53mm]{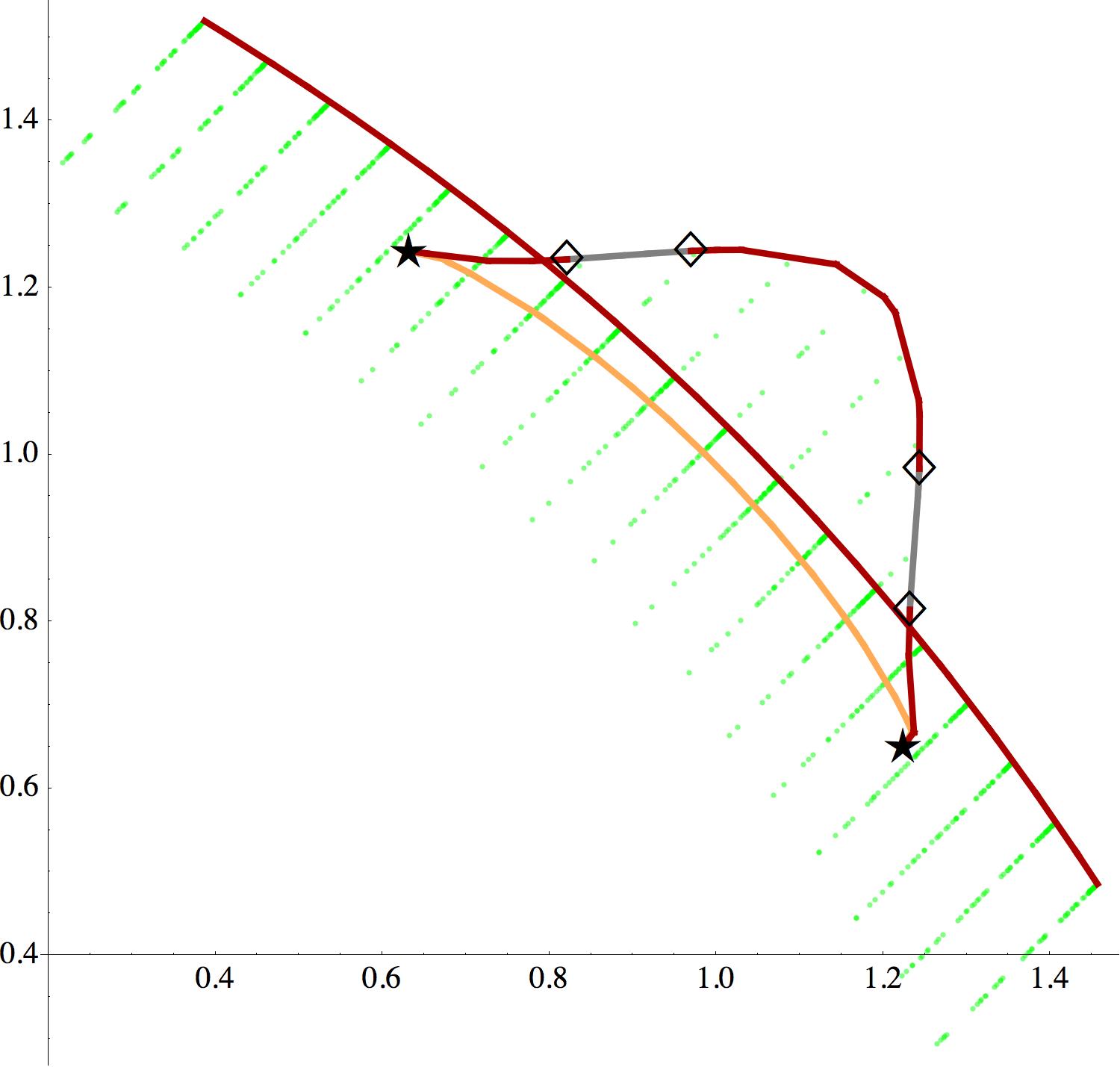} \\
(a) & (b)\\
\includegraphics[width=60mm]{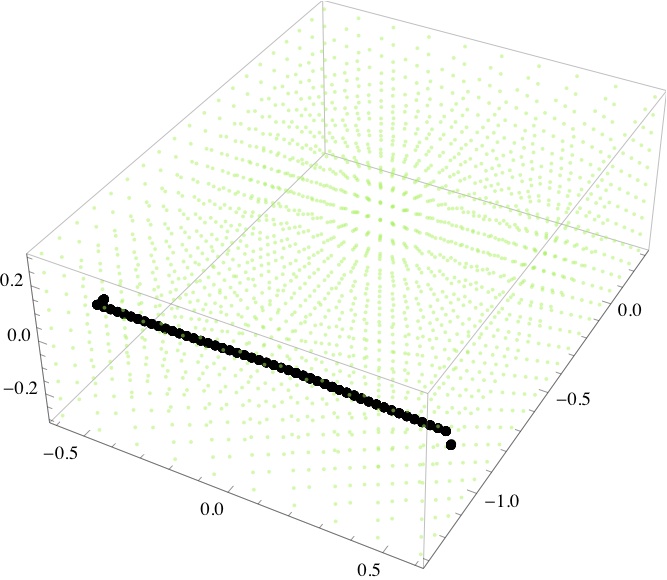}&
\includegraphics[width=45mm]{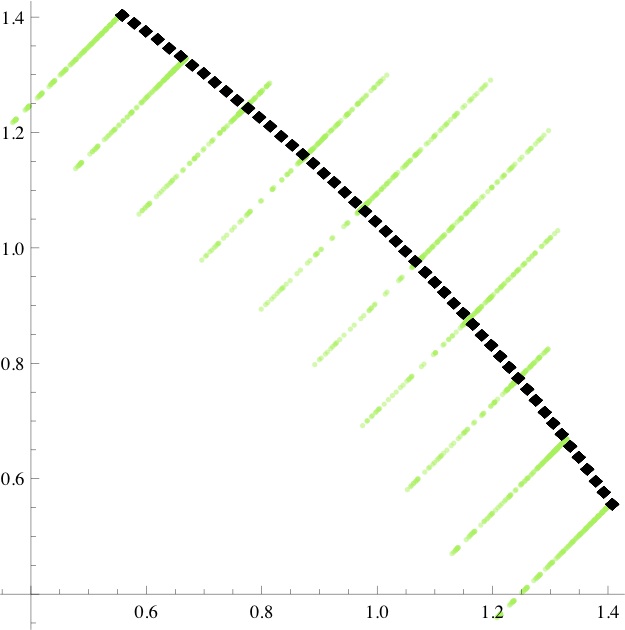} \\
(c) & (d)
\end{tabular}
\caption{Example \ref{ex:locglob}. Panel (a): singular (gray), Pareto critical (orange) and Pareto stable (red) sets in the problem domain. Green dots mark the nodes of the starting regular grid defining the tessellation. Octahedrons mark points separating critical and non critical branches. Spheres separate stable from unstable branches, i.e., mark cusps. Panel (b): image of singular and Pareto sets. Diamonds separates critical from non critical branches while stars mark the cusps. Panel (c)--(d): results obtained running the commercial implementation of NBI--AFSQP  available in \textup{\sf modeFRONTIER}$^\circledR$, courtesy of E. Rigoni. Small green points are a starting regular grid, while marked points are the solutions of the 50 NBI subproblems.
}
\label{fig:locglob}
\end{figure}
\end{esem}

\begin{esem} \label{ex:ncv6}
The following 6 dimensional example is a regularization of the third of the ZDT problems \cite{Deb:1999p2938}, which has degenerate second derivatives. The Pareto fronts of original and modified problems correspond each other in output space. We used a Delaunay tessellation defined on 300 randomly generated points. The results are presented in figure \ref{fig:ncv6}. Critical and merely singular branches are correctly represented. Note that the critical branches are correctly marked as unstable, being minima.
\begin{equation}
   \begin{split}
u1(x_1,\dots,x_6)& = x_1, \\
u2 (x_1,\dots,x_6)& = 1 - \sqrt{x_1} - x_1 \sin (10 \pi x_1) + x_2^2 +\dots+ x_6^2, \\
& x_1\in [0.1,0.425],\qquad x_2,\dots,x_6\in [-0.16,0.16].
   \end{split}
\end{equation}	
\begin{figure}
\centering
\includegraphics[width=60mm]{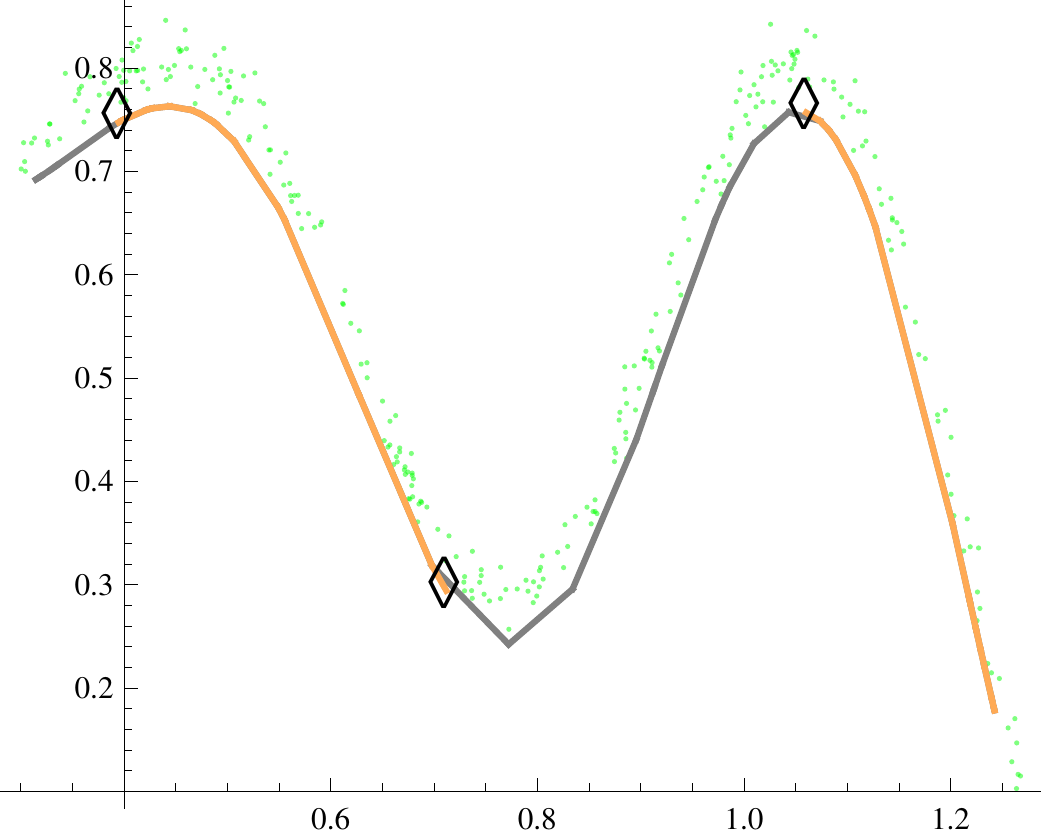} 
\caption{Example \ref{ex:ncv6}. Panel (a): image of the singular (gray), Pareto critical (orange) sets for regularized ZDT3. Green dots mark the nodes of the starting random distribution of points defining the tessellation. }
\label{fig:ncv6}
\end{figure}
\end{esem}

\subsection{Three functions examples}
\begin{esem} \label{ex:3by3}
The simplest nontrivial non degenerate example we can build in the three dimensional case is composed by three negative definite 2nd order polynomial functions $f_j(x), j=1,2,3$. Additionally, we introduce a small non polynomial perturbation.  
\begin{equation}
\begin{split}
 f_j (x) = \tonde{x-C_j}^\top \cdot
		\begin{pmatrix}
			- \alpha_{j,1} & 0 & 0 \\
			0 &- \alpha_{j,2} & 0 \\
			0 & 0 &- \alpha_{j,3}  
		\end{pmatrix}\cdot
		(x-C_j), \quad j=1,2,3,\\
	\trevec {u_1(x)} {u_2(x)} {u_3(x)} := \trevec{f_1(x)}{f_2(x)}{f_3(x)} + \trevec{0}{\beta_2 \sin\tonde{\frac{\pi}{\gamma_2}(x + y)}}{\beta_3 \cos\tonde{\frac{\pi}{\gamma_3}(x - y)}}.
\end{split}
\end{equation}	
Where $x=(x_1,x_2,x_3)^\top\in\R^3$, $\alpha_{j,i}>0$, $i,j=1,2,3$, $C_1, C_2, C_3\in\R^3$ are distinct, non collinear points, while $\beta_j, \gamma_j$ are real numbers. 
In the generic case the singular set is an hypersurface of $\R^3$, while the critical set $\theta$, which is stable, is diffeomorphic to a triangle, i.e., $\theta$ is a compact connected manifold with boundary and three corners, corresponding to the minima of the three functions $u_1,u_2,u_3$. 
See Figure \ref{fig:3by3}(a).
\end{esem}
\begin{esem}\label{ex:3by3ncv}
	We break the convexity of the previous example by adding a secondary maximum to the first function. We define a further negative definite, $2^{nd}$ order polynomial  $f_4(x)$  and set $u(x)$ as:
	\begin{equation}
	\trevec {u_1(x)} {u_2(x)} {u_3(x)} := \trevec{f_1(x)}{f_2(x)}{f_3(x)} + \trevec{\beta_1 \exp\tonde{\frac{1}{\gamma_1} f_4(x)}}{\beta_2 \sin\tonde{\frac{\pi}{\gamma_2}(x + y)}}{\beta_3 \cos\tonde{\frac{\pi}{\gamma_3}(x - y)}}.
	\end{equation}
The main portion of the Pareto set is slightly deformed while a new branch appears. 
In Figure \ref{fig:3by3}(b) is shown the resulting Pareto critical set $\theta$ obtained by iterative application of Algorithm \ref{alg:firstorder} as described in section \ref{sec:iterative}.

\end{esem}
\begin{figure}
\centering
\begin{tabular}{cc}
\includegraphics[width=60mm]{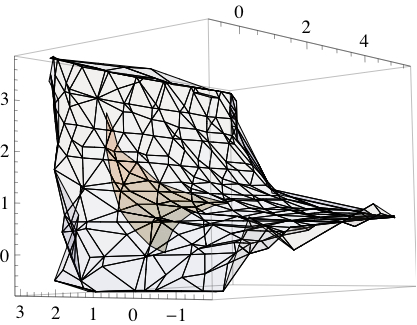}&\includegraphics[width=60mm]{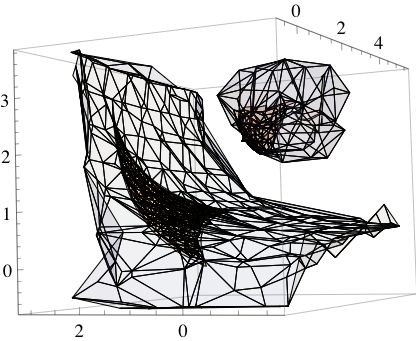} \\ 
(a) & (b)
\end{tabular}
\caption{(a): critical set $\theta$ of example \ref{ex:3by3}, (b): critical set $\theta$ of example \ref{ex:3by3ncv}}
\label{fig:3by3}
\end{figure}

%
%
\subsection{A constrained example} We briefly sketch here an adaptation of 
Algorithm \ref{alg:firstorder} to the case of equality constraints. Next we illustrate a simple application. 
Let $W:=\set{x\in\R^n\taleche  g(x) =0}$, where $g:\R^n\To\R^{n-d}$ is a smooth function such that $\partial g\eval{W}$ has maximum rank.
\begin{algorithm}[H]
 \caption {\it Equality constraint case for the 1st order algorithm}
 \label{alg:equalityconstraint}
\begin{algorithmic}[1]
 \STATE Determine a piecewise linear approximation $\widehat W$ of $W$, with nodes $P_1,\dots,P_n$ 
 \FORALL{ simplex $\Delta$ in the tessellation $\widehat W$,} 
 \STATE determine a piecewise linear approximation of the singular, critical and stable sets possibly crossing the simplex.  This is done on the basis of the projections of the gradients of the $u_j$s on the tangent space to $W$. In principle, a basis for $T W$ should be chosen respecting the orientation. 
 	\FORALL{ Node $P$ of the simplex $\Delta$} 
	\STATE compute $D g(P)$ and project $\grad u_j(P)$ on $\ker Dg(P)$ via $\pi_{Dg(P)}$. 
	\STATE compute an independent set of minors for the matrix $(\pi \grad u_1(P),\dots,\pi \grad u_m(P))$ or, equivalently,
	\STATE compute an independent set of minors for the matrix $(\grad g_1,\dots,\grad g_{n-d},\grad u_1, \dots,\grad u_m)$
	\ENDFOR
\STATE determine if all minors vanish inside the simplex $S$, and in that case locate $\widehat \Sigma$ via inverse linear interpolation
\STATE estimate $\lambda_j$ and determine the critical set $\widehat \theta$ as in Algorithm \ref{alg:firstorder}
\ENDFOR
\STATE eventually glue together adjacent portions of $\widehat \Sigma$ and $\widehat \theta$.
\end{algorithmic}
\end{algorithm}

\begin{esem}\label{ex:sphereproj}
Maybe the simplest example of constrained problem is when $W=\mathbb{S}^2$ and the objectives are the first two coordinates, $u_1(x_1,x_2,x_3):=x_1$, $u_2(x_1,x_2,x_3):=x_2$.\footnote{This example is also discussed in \cite{Melo:1976sj}.} 
Explicit algebraic computation give that the singular set $\Sigma$ is the equator of the sphere, where the two curvilinear segments where $x_1x_2>0$ are the critical set $\theta$, as illustrated in figure \ref{fig:sphereproj}(a).
By applying algorithm \ref{alg:equalityconstraint}, we start by approximating the sphere by an icosahedron. At every node $P= (x_1,x_2,x_3) $, $\frac{1}{2}\tonde{x_1^2+x_2^2+x_3^2 -1}=0$, we have $Dg(P)= (x_1,x_2,x_3) $, therefore the projections of the gradients of $u_j$ are $(1-x_1^2, -x_1x_2, -x_1x_3)$ and $(-x_1x_2, 1-x_2^2, -x_1x_3)$. The singular set $\Sigma$ passes through the triangles where the pair of vectors $\pi \grad u_1$ and $\pi \grad u_2$ change orientation in the tangent plane to $\mathbb{S}^2$. It is equivalent then to compute the determinant of the matrix which columns are $\grad g$, $\grad u_1$ and $\grad u_2$ and to determine the line along which it vanishes. This gives exactly the ``equator'' of the icosahedron. The signs of the $\lambda_j$ depend on the sign of the scalar product among $\pi\grad u_j$, again giving as turning points the intersections with the axes. The results are summarized in figures \ref{fig:sphereproj}(a) and (b).
\end{esem}
\begin{figure}
\centering
\begin{tabular}{cc}
\includegraphics[width=60mm]{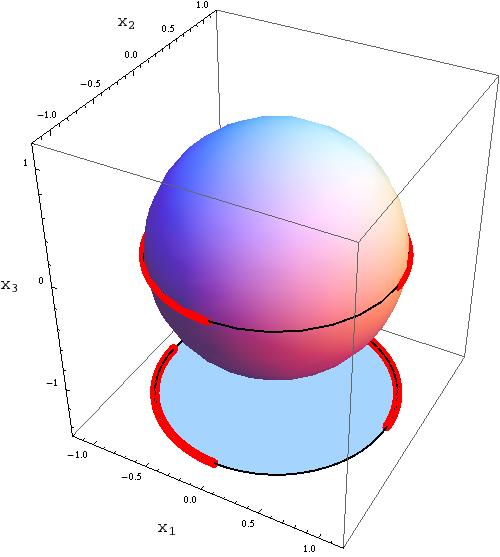}&\includegraphics[width=60mm]{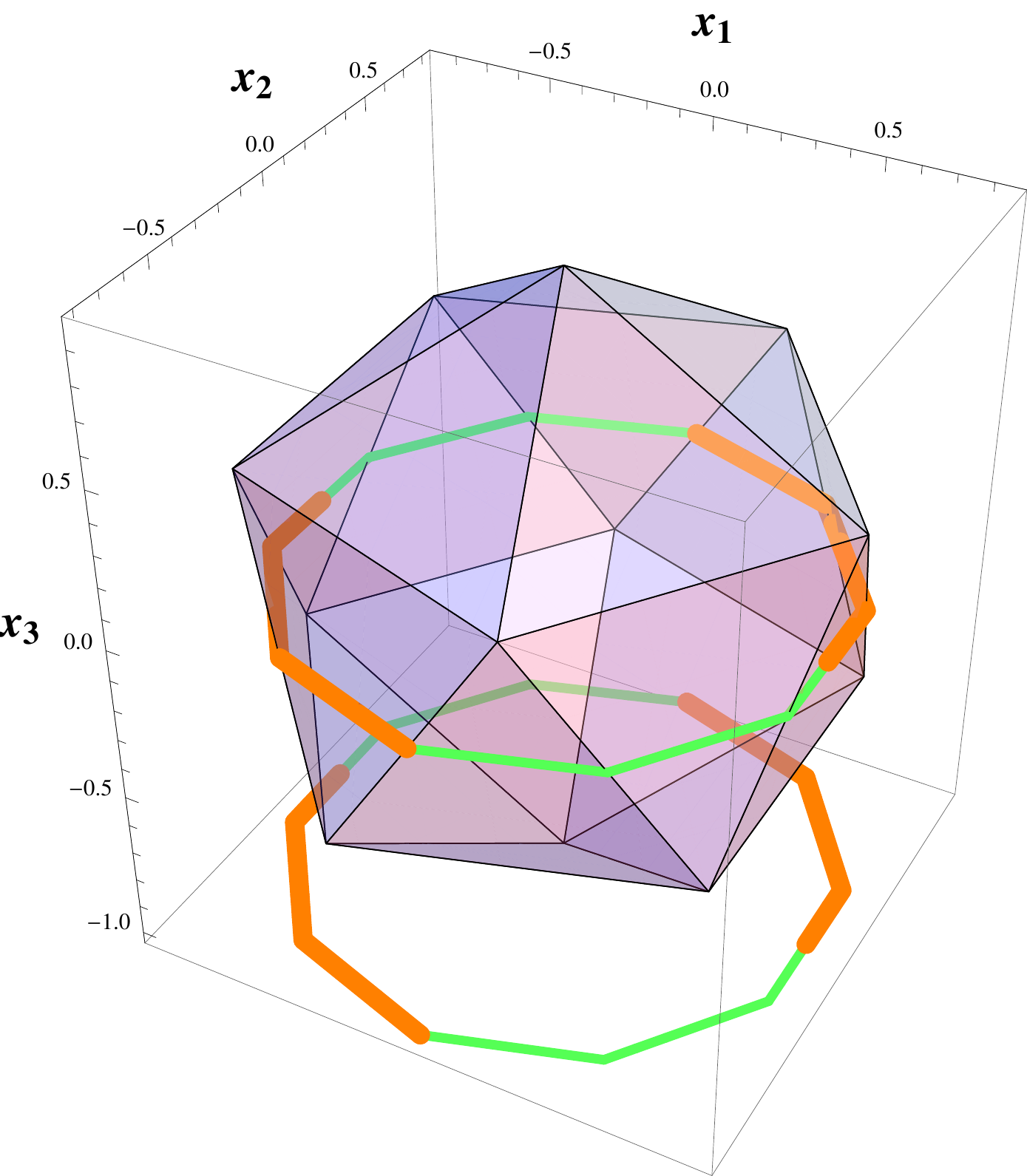} \\ 
(a) & (b)
\end{tabular}
\caption{(a): Singular and critical sets determined analytically for example \ref{ex:sphereproj} (b): piecewise linear approximation of the sphere, of the singular set (green solid curves) and of the critical set (orange and thicker curves)}
\label{fig:sphereproj}
\end{figure}

\section{Iterative schemes} 
\label{sec:iterative}
The previously presented approach defines an approximation of the Pareto optimal set given any distribution of points in the domain. Here we propose and discuss an iterative 
scheme. At every step a selection of points from the approximated Pareto optimal set is added to the dataset $\mathcal{D}$, the gradients in the new points are evaluated, the tessellation is updated and a refined approximation of the Pareto set is built.
The desired effect is obviously to get closer and closer to the actual optimal set, but an efficient strategy should produce an as uniform as possible discretization of the optimal set.

A na\"{\i}ve approach would suggest to insert in the set of the candidates for evaluation all of the nodes of the complexes, i.e., all the stable admissible vertices computed and all of the boundary points, both for criticality and stability. 
Nevertheless, a glance at the examples of the previous section reveals that the sizes of the optimal complexes cover a wide distribution, in particular the patches result very small if $\Sigma$ passes close to tessellation nodes.
Moreover, little experience shows that large patches are reduced sensibly slowly if none of their internal points is introduced. 
With this criteria in mind we introduce an iterative scheme for the case of two functions.

\subsection{Two functions iterative scheme}

In the two functions case the Pareto optimal set is a one dimensional manifold with boundary, i.e., a collection of curved intervals. The discrete approximation is a collection of polygonal curves. 
For every interval a sequence of candidate points equally spaced along the polygonal curve is extracted.
The number of points is chosen equal to the number of segments, so that approximately every triangle containing optimal points is split as close as possible to the optimal set.

\subsection{Higher number of functions}
It seems reasonable to take into account of the stratified structure of $\theta$ in the design of an iterative strategy. In fact, strata should be filled as uniformly as possible, where the uniformity is determined according to the $k$--dimensional measure, if $k$ is the dimension of the stratum. So, taking for instance the situation of example \ref{ex:3by3}, corners' approximations are re--evaluated at each iteration, uniformly spaced points are taken along boundary lines, exactly as in the two functions case, while internal points should be distributed proportionally to the area of the triangles and polygons composing $\widehat\theta$. This is more difficult to be defined precisely. 
Indeed, the problem of uniformly filling a general $n$--dimensional region is a long--time crucial issue for   statistical applications \cite{Santner:2003wm}. Furthermore, in our problem we have to fill uniformly a general $n$--dimensional \emph{manifold}, thus we have somehow to take into account of the effects of the curvature on the measure of the volumes. 

Taking inspiration from a \emph{Design of Experiments} strategy called \emph{maximin distance design} \cite{Johnson:1990kc} we proceed as described in Algorithm \ref{alg:uniform_filling}. 
\begin{algorithm}
\caption{Uniformly filling a simplicial complex}
\label{alg:uniform_filling}
\begin{algorithmic}[1]
\STATE Tessellate in simplexes the polytopes of the mesh
\STATE Build the adjacency lists of the simplexes
\STATE Evaluate the volume of each simplex
\STATE For every simplex define the accumulated volume as the sum of its volume and the volume of the adjacent simplexes
\STATE Pick the simplex with the maximum accumulated volume
\STATE Add to the candidates stack the center of mass of this maximal simplex 
\REPEAT
	\STATE Recompute the accumulated volumes excluding the already picked simplexes
\UNTIL the desired number of candidate points is collected
\end{algorithmic}
\end{algorithm}
This algorithm, because of point 6, can lead to long and thin simplexes and to numerical instabilities when iterated many times. This problem can be tackled by the application of mesh improvement strategies, as described below for the case of two dimensional domains. However, to the author's knowledge, general dimension mesh improvement strategies are still not available at present. 
\subsection{Stopping criteria} Analogously to gradient based methods of single function optimization (nonlinear conjugate gradient, Newton and Newton--like methods), a stopping criterion could be based on the magnitude of the minors $M_1, \dots,M_r$ computed in the points of the last iteration. The magnitude of the minors is analogous to the magnitude of the gradients for single objective optimization. 

In fact, we could define a different iterative strategy taking the rule of subdividing only stable critical triangles contained in simplexes where the minors are larger than a prescribed threshold. 

\subsection{Application}
We show the behavior of the iterative scheme described above applied to the mapping in Example \ref{ex:noncv}. At each iteration we generate a number of evenly spaced points along the approximate stable Pareto critical set. 
In order to exhibit the claimed quadratic convergence, it is necessary to sample the approximated optimal set by quadratically finer  intervals, i.e., comparable to the precision gained. As a result 
 the density of points will grow exponentially w.r.t. the number of iterations.  
Such a density of points rapidly deteriorates the mesh quality, i.e., skinny triangles suddenly appear leading to numerical instability.  Thus, at each iteration, a number of extra nodes (namely, the circumcenters of the most skinny triangles) should be introduced in the mesh in order to produce a nicely grading.  At this extent we have coupled our method with Ruppert's algorithm, as implemented in the triangular mesh refinement software \textsc{triangle} by J. R. Shewchuk
\cite{Shewchuk:1996xh,Shewchuk:2002wa}.


Already at the fifth iteration the triangulation starts to suffer from numerical instability, thus we consider $\widehat\theta_S^{(4)}$ generated at the fourth  iteration as the 
optimum and evaluate the Hausdorff distances between $\widehat\theta_S^{(i)}$ and $\widehat\theta_S^{(4)}$, for $i=1,\dots,3$. As in can be seen in Figure \ref{fig:haus}, panel (a), the Hausdorff distances between the approximated Pareto sets and the numerical optimum converges superlinearly. For reference also the convergence behavior of the maximum and the mean minors magnitude are reported. 

In Figure \ref{fig:progression2D} is illustrated how the triangulation and the representation of the Pareto set evolves from one iteration to the subsequent. 

In Figure \ref{fig:3by3prog}, the three dimensional problem of Example \ref{ex:3by3} is tackled by the procedure described in Algorithm \ref{alg:uniform_filling}. The algorithm has been applied by introducing only a small number ($\sim$10) of new points on the sites with the largest magnitude of the minors. In such a way it was possible to iterate 70 times the scheme 
reaching a very small magnitude for the minors. 

Because of the mentioned exponentially growing number of samples necessary to exhibit quadratic convergence speed for the iterative scheme, the experiment described for the two dimensional case becomes prohibitive in three dimensions. 

On the other hand, a superlinear precision can be verified as well by means of a sequence of approximations obtained from a  progressively finer regular meshes, corresponding to a sequence of mesh sizes $s=2.8,\dots,0.4$.  
Because the Hausdorff distance among the first $s$--approximation and the optimal set is already comparable to the largest mesh size of the optimal set, we analyze the sequence of average distances  between a point  of one set from the triangles of the other set, instead of considering the maximum distances. 
These average distances decrease faster than linearly as it can be seen by plotting the ratio of distances and mesh sizes versus the mesh sizes, 
as reported in Figure \ref{fig:haus}, panel (b).
\begin{figure}
\centering
\begin{tabular}{cc}
\raisebox{25mm}{\parbox{15mm}{\it \footnotesize Step 1.\\ 29 pts \\ (from 66)}} & \includegraphics[width=105mm]{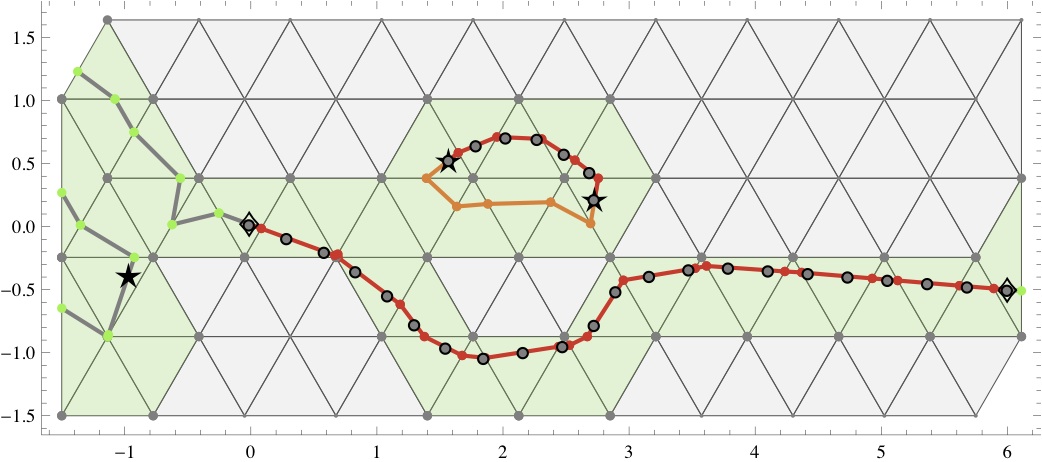} \\  
\raisebox{25mm}{\parbox{15mm}{\it \footnotesize Step 2.\\ 56 pts\\ (from 95)}} &  \includegraphics[width=105mm]{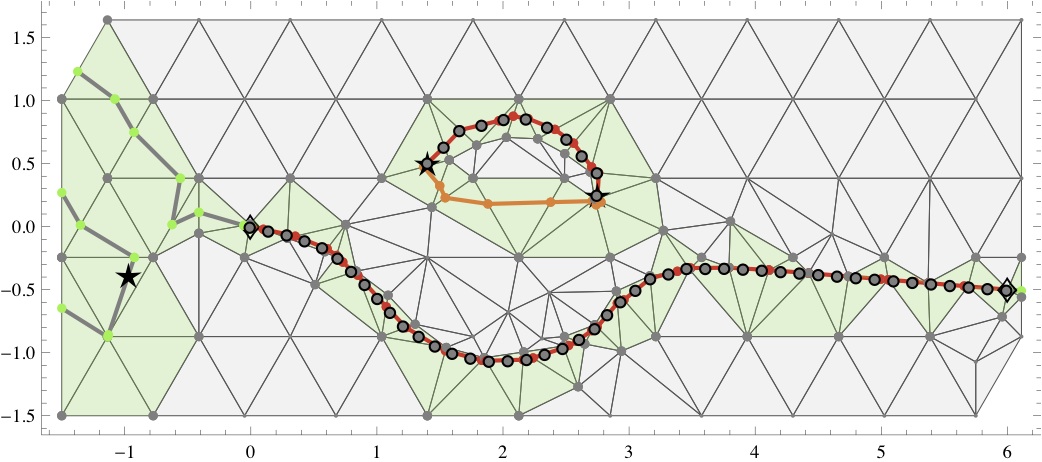} \\
\raisebox{25mm}{\parbox{15mm}{\it \footnotesize Step 3.\\ 302 pts\\ (from 168)}} & \includegraphics[width=105mm]{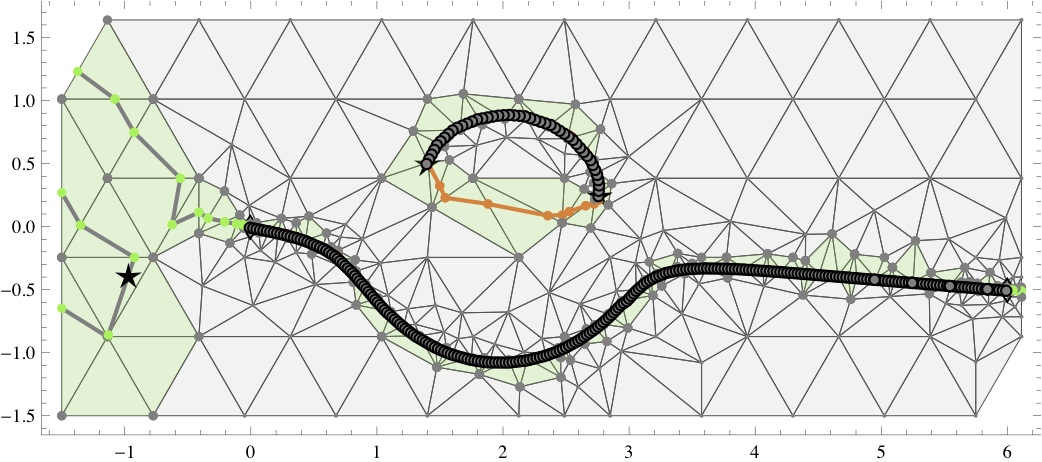} \\
\raisebox{25mm}{\parbox{15mm}{\it \footnotesize Step 4.\\ 2977 pts\\ (from 594)}} & \includegraphics[width=105mm]{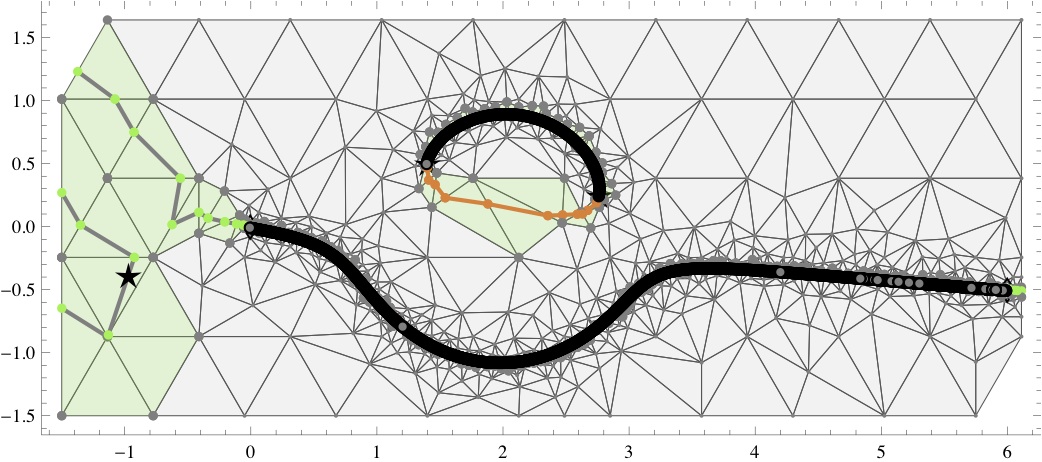}   
\end{tabular} 
\caption{Iterative scheme for the mapping in Example \ref{ex:noncv}.}
\label{fig:progression2D}
\end{figure}

%
\begin{figure}
\centering
\begin{tabular}{cc}
\includegraphics[width=50mm]{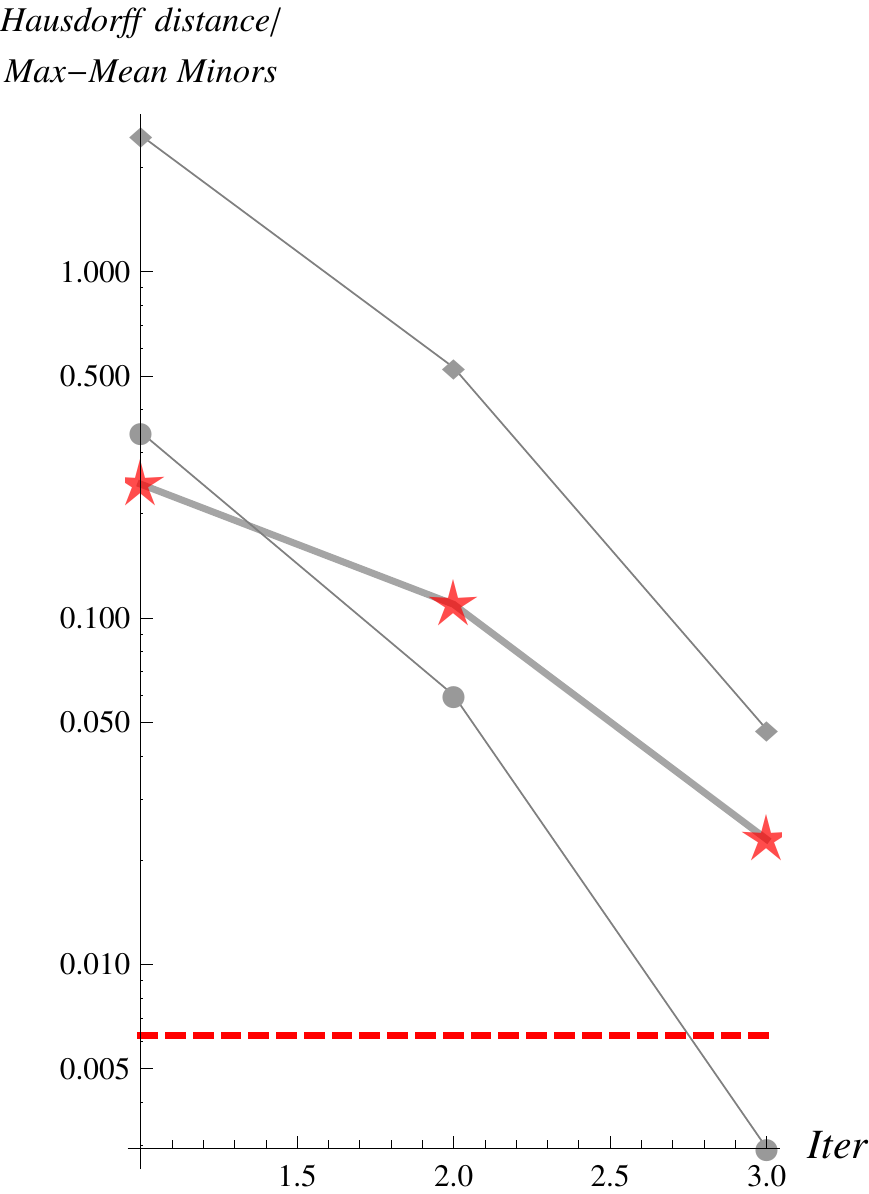} &
\parbox[b]{66mm}{\includegraphics[width=65mm]{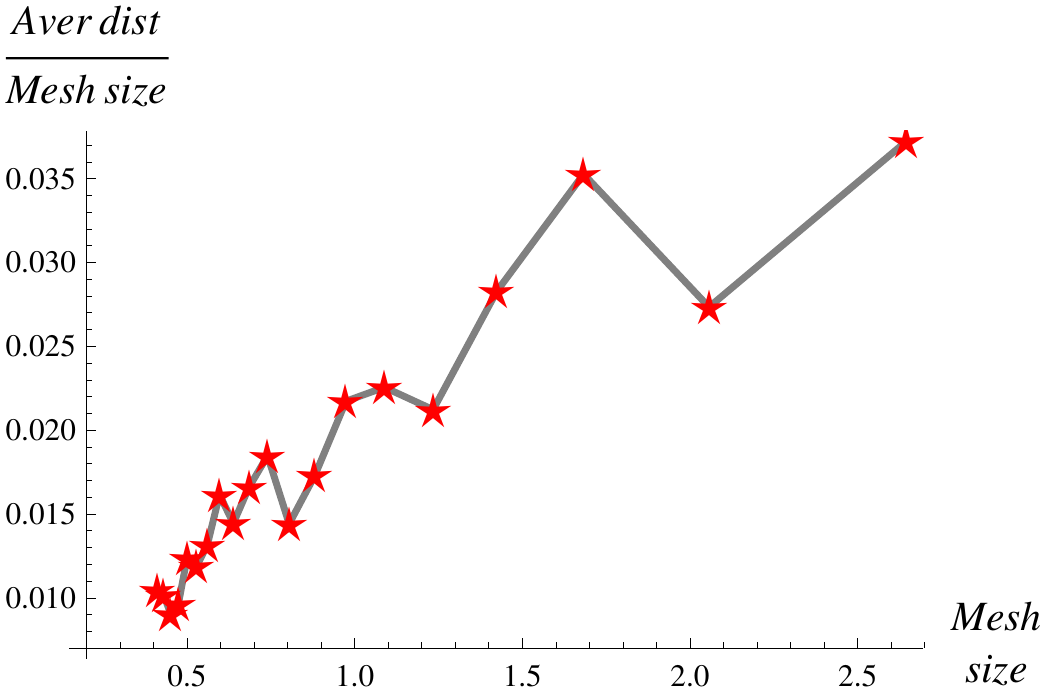}}\\
\it \footnotesize (a) Example  \ref{ex:noncv}. & \it \footnotesize (b) Example  \ref{ex:3by3}.
\end{tabular}
\caption{ Convergence behavior for iterative schemes applied to Examples  \ref{ex:noncv}  and \ref{ex:3by3}. 
Panel (a): Iterative scheme applied to Example \ref{ex:noncv}. Stars represent the Hausdorff distance between the approximated Pareto set at each iteration and the Pareto set obtained at the 4--th iteration, which is employed as an optimum. 
Diamonds and circles represent respectively the maximum and the mean absolute value of the minors of the Jacobian matrix computed on the points of the approximated Pareto set. Log scale reveals the superlinear convergence behavior.
Horizontal dashed line represents the mesh size of the numerical optimum. 
Panel (b): Algorithm \ref{alg:firstorder} applied to Example \ref{ex:3by3} using progressively finer regular meshes. 
Stars represent the average distance between a node of the optimal set and the triangles of the approximation and viceversa.
The ratio between the distance and the the mesh size decreases faster than linearly according to Theorem \ref{teo:quadraticsing}.
}
\label{fig:haus}
\end{figure}


\begin{figure}
\centering


\begin{tabular}{cc} 
\includegraphics[width=45mm]{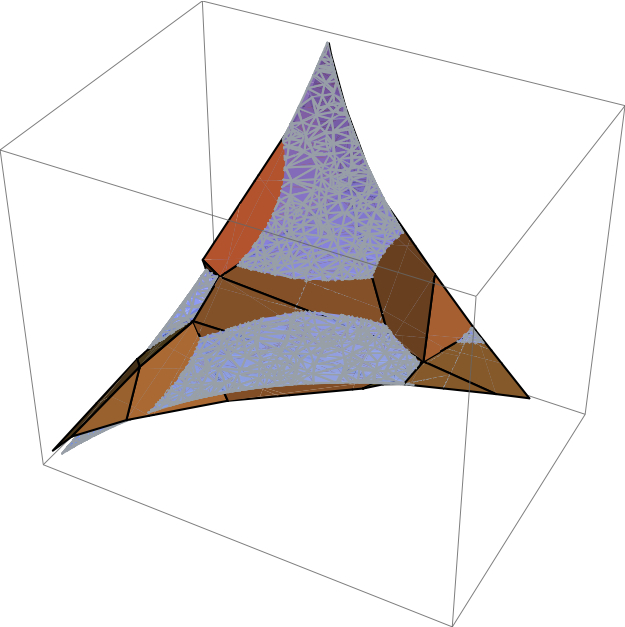} 
 & \includegraphics[width=45mm]{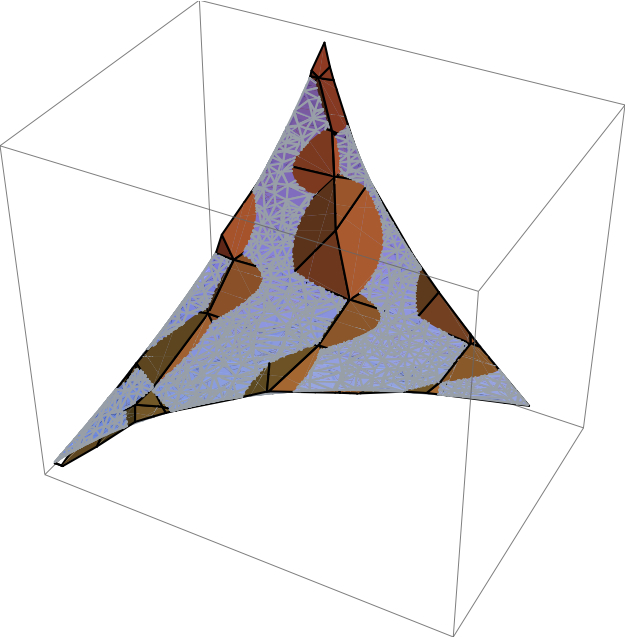}\\
\it\footnotesize Mesh size $= 2.3$ 
& \it \footnotesize Mesh size $= 1.5$
\end{tabular}\\
\begin{tabular}{c} 
\includegraphics[width=85mm]{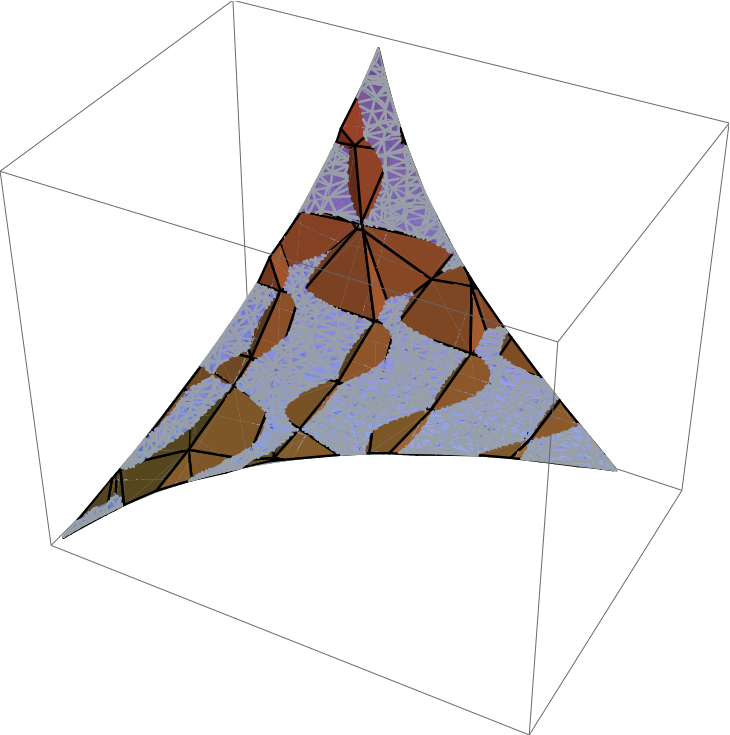}\\
\it \footnotesize Mesh size $=0.7$
\end{tabular}
\caption{Algorithm \ref{alg:firstorder} applied to Example \ref{ex:3by3} for progressively finer meshes. The blue surface is the numerical optimum obtained with long application of the iterative scheme.}
\label{fig:3by3prog}
\end{figure}

\section{Conclusions and perspectives}
We have presented a novel multiobjective optimization method which exploits
the manifold structure underlying the set of Pareto optimal points. Global analysis seems the proper setting 
where those structures arise and can be studied. 
We approximate Pareto sets via simplicial complexes, specializing simplicial pivoting techniques for detecting the singular manifold $\Sigma$ and successively cutting out critical and stable subsets $\theta$ and $\theta_S$. 
By contrast, most of the available strategies are aimed at producing a scatter of optimal points which images should be evenly distributed. We have illustrated some generic situation where this program could not be successfully completed via such point--wise strategies, because of nonconvexities of the functions.
Adopting the Hausdorff measure, Newton--type estimates lead to quadratic convergence in a set--wise sense.

Because of its global character, the method proposed here is demanding. Delaunay tessellations, in particular, are defined for every possible input dimension, but are numerically workable only for small dimensional cases.
The theory of singularities of mappings also highlights further limitations encountered when dealing with a large number of functions. Lastly, we have everywhere assumed the differentiability of the functions. Therefore the method is not suitable for non--smooth optimization, instead it is supposed to be applicable also via smooth surrogate functions, when approximations are consistent with the functions at hand. 
Possible extensions of the algorithms described in this paper are conditioned by the issues enumerated below. 

\subsection{The curse of dimensionality} The first problem one encounters when trying to apply this algorithms to industrial strength problems is the limitations to the input dimension. The whole procedure is 
based on a Delaunay tessellation of the input domain, which complexity grows exponentially with dimension. As pointed out for instance in the \texttt{qhull} documentation \cite{Barber:1996mk}, building the convex hull of a 9-hypercube is computationally exhaustive.
Analogous limitations are encountered in global optimization, where the search for optima in high dimensional domains cannot realistically be performed on real case problems. Indeed, typically, global search algorithms are rarely tested and compared over dimensions larger than 5 (see 
\cite{
Jones:1998et,Jones:2001sw,Khompatraporn:2005cr,
Pinter:2009fb,
Schonlau:1998jh,
Sergeyev:2006bu,
Wu:2005lp}.
This problem is structural and cannot be resolved 
by augmenting the computational resources.  
Therefore, the presented algorithms are best suited for low dimensional problems. 
In fact, the \emph{curse of dimensionality} is a strong motivation for reflecting  carefully on the necessity of introducing extra input variables 
when tackling new problems and designing experiments. 
A possible exit strategy  could be screening the input variables \cite{Schonlau:2005lp,
Sobol:2009ph
}.
This practice can be surprisingly successful, because usually \emph{sparsity of effects} occurs, revealing a pronounced  hierarchy among input variables, 
leading to sensible simplification of the problem formulation.
\footnote{The sparsity of effects is an empirical law stating that in a generic physical experiment one usually observes that the $80\%$ of the effects are due to the $20\%$ of the factors. Related phenomena are that the first order contributions are the most important, while higher order contributions decay fast. Finally one observes that the largest interactions (second order contributions) are combination of the strongest factors. See for instance 
\cite{Wu:2000vo}.}

Alternatively, as described in the recent paper \cite{Allgower:2002p2552}, it is possible to redefine any problem in $n$-dimensional space in an equivalent problem in a linear subspace of dimension $2(m-1) +1$, if $m$ is the number of objectives. This is because the singular set is an $(m-1)$-manifold, and by Whithney's embedding theorem, in the compact case, almost all projections on linear   $2(m-1) +1$ dimensional subspaces are diffeomorphisms. This would mean that bi--objective problems 
could be equivalently discussed in a 3 dimensional domain, 3 objectives would require only 5 input variables, and so on. 
This would dramatically reduce the computational burden of the tessellations involved.


\subsection{Surrogate models}
In industrial applications, when the objective functions at hand could be non differentiable, or may be computationally too expensive, preventing the computation of derivatives, we figure that the applicability of the algorithm proposed here will be significantly extended by using surrogate models. There exists an extensive literature developed in recent years on this subject (see
\cite{Jones:1998et,Jones:2001sw,Sacks:1989pz,Santner:2003wm}
 and the references therein), also with specific applications to multiobjective optimization 
 \cite{Knowles:2006fh,Knowles:2008px}

The procedures of this paper can be adapted applying the Algorithms \ref{alg:firstorder} and \ref{alg:secorder} to a surrogate model $\tilde u$ fitted to the values of the true functions $u$ computed on the given data points. On the outcoming candidate points, new evaluations of $u$ are to be computed, and a new surrogate model fitted to the increased dataset. 
This reduces the computational effort for computing derivatives and furthermore prevents premature stopping of the optimization process due to accidental failure of function evaluation at some data point. 
Again, the convergence to the Pareto sets of the true functions is guaranteed via global analysis.

\bibliographystyle{mysiam}
\bibliography{synthesis}

\begin{thebibliography}{10}

\bibitem{Allgower:1980p2815}
{\sc E.~L. Allgower and K.~Georg}, {\em Simplicial and continuation methods for
  approximating fixed points and solutions to systems of equations}, SIAM Rev.,
  22 (1980), pp.~28--85.

\bibitem{Allgower:1989p2804}
\leavevmode\vrule height 2pt depth -1.6pt width 23pt, {\em Estimates for
  piecewise linear approximations of implicitly defined manifolds}, Appl. Math.
  Lett., 2 (1989), pp.~111--115.

\bibitem{Allgower:2000p2573}
\leavevmode\vrule height 2pt depth -1.6pt width 23pt, {\em Piecewise linear
  methods for nonlinear equations and optimization}, Journal of Computational
  and Applied Mathematics, 124 (2000), pp.~245--261.

\bibitem{Allgower:1991p2654}
{\sc E.~L. Allgower and S.~Gnutzmann}, {\em Simplicial pivoting for mesh
  generation of implicitly defined surfaces}, Comput. Aided Geom. Design, 8
  (1991), pp.~305--325.

\bibitem{Allgower:1985p2832}
{\sc E.~L. Allgower and P.~H. Schmidt}, {\em An algorithm for piecewise-linear
  approximation of an implicitly defined manifold}, SIAM J. Numer. Anal., 22
  (1985), pp.~322--346.

\bibitem{Allgower:2002p2552}
{\sc E.~L. Allgower and A.~J. Sommese}, {\em Piecewise linear approximation of
  smooth compact fibers}, J. Complexity, 18 (2002), pp.~547--556.

\bibitem{Arnold:1968fx}
{\sc V.~I. Arnol{$'$}d}, {\em Singularities of smooth mappings}, Russian
  Mathematical Surveys, 23 (1968), pp.~1--43.

\bibitem{Barber:1996mk}
{\sc C.~B. Barber, D.~P. Dobkin, and H.~T. Huhdanpaa}, {\em The quickhull
  algorithm for convex hulls}, ACM Trans. on Mathematical Software, 22 (1996),
  pp.~469--483.

\bibitem{Benson:1997p2942}
{\sc H.~Benson and S.~Sayin}, {\em Towards finding global representations of
  the efficient set in multiple objective mathematical programming}, Naval
  Research Logistics, 44 (1997), pp.~47--67.

\bibitem{Bowyer:1981ud}
{\sc A.~Bowyer}, {\em Computing {D}irichlet tessellations}, Comput. J., 24
  (1981), pp.~162--166.

\bibitem{Das:1998zh}
{\sc I.~Das and J.~E. Dennis}, {\em Normal-boundary intersection: a new method
  for generating the pareto surface in nonlinear multicriteria optimization
  problems}, SIAM J. Optim., 8 (1998), pp.~631--657 (electronic).

\bibitem{Melo:1976sj}
{\sc W.~de~Melo}, {\em On the structure of the {P}areto set of generic
  mappings}, Bol. Soc. Brasil. Mat., 7 (1976), pp.~121--126.

\bibitem{Melo:1976xd}
\leavevmode\vrule height 2pt depth -1.6pt width 23pt, {\em Stability and
  optimization of several functions}, Topology, 15 (1976), pp.~1--12.

\bibitem{Deb:1999p2938}
{\sc K.~Deb}, {\em Multi-objective genetic algorithms: Problem difficulties and
  construction of test problems}, Evolutionary Computation, 7 (1999),
  pp.~205--230.

\bibitem{Debreu:1976ta}
{\sc G.~Debreu}, {\em Regular differentiable economies}, The American Economic
  Review, 66 (1976), pp.~280--287.

\bibitem{Debreu:1993if}
\leavevmode\vrule height 2pt depth -1.6pt width 23pt, {\em Stephen {S}male and
  the economic theory of general equilibrium}, in From {T}opology to
  {C}omputation: {P}roceedings of the {S}malefest ({B}erkeley, {CA}, 1990),
  Springer, New York, 1993, pp.~131--146.

\bibitem{Fliege:2009kc}
{\sc J.~Fliege, L.~M.~G. Drummond, and B.~F. Svaiter}, {\em Newton's method for
  multiobjective optimization}, SIAM J. Optim., 20 (2009), pp.~602--626.

\bibitem{Gourion:2008rr}
{\sc D.~Gourion and D.~T. Luc}, {\em Generating the weakly efficient set of
  nonconvex multiobjective problems}, J. Global Optim., 41 (2008),
  pp.~517--538.

\bibitem{Guillemin:1974hs}
{\sc V.~Guillemin and A.~Pollack}, {\em Differential topology}, Prentice-Hall,
  Jan 1974.

\bibitem{Johnson:1990kc}
{\sc M.~E. Johnson, L.~M. Moore, and D.~Ylvisaker}, {\em Minimax and maximin
  distance designs}, Journal of Statistical Planning and Inference, 26 (1990),
  pp.~131--148.

\bibitem{Jones:2001sw}
{\sc D.~R. Jones}, {\em A taxonomy of global optimization methods based on
  response surfaces}, J. Global Optim., 21 (2001), pp.~345--383.

\bibitem{Jones:1998et}
{\sc D.~R. Jones, M.~Schonlau, and W.~J. Welch}, {\em Efficient global
  optimization of expensive black-box functions}, J. Global Optim., 13 (1998),
  pp.~455--492.

\bibitem{Khompatraporn:2005cr}
{\sc C.~Khompatraporn, J.~D. Pint{\'e}r, and Z.~B. Zabinsky}, {\em Comparative
  assessment of algorithms and software for global optimization}, J. Global
  Optim., 31 (2005), pp.~613--633.

\bibitem{Knowles:2006fh}
{\sc J.~Knowles}, {\em Parego: a hybrid algorithm with on-line landscape
  approximation for expensive multiobjective optimization problems}, IEEE
  Trans. Evolutionary Computation, 10 (2006), pp.~50--66.

\bibitem{Knowles:2008px}
{\sc J.~Knowles and H.~Nakayama}, {\em Meta-modeling in multiobjective
  optimization}, Multiobjective Optimization,  (2008), pp.~245--284.

\bibitem{Levine:1976bs}
{\sc H.~Levine}, {\em Stable maps: an introduction with low dimensional
  examples}, Bol. Soc. Brasil. Mat., 7 (1976), pp.~145--184.

\bibitem{Luc:2005sp}
{\sc D.~T. Luc, T.~Q. Phong, and M.~Volle}, {\em Scalarizing functions for
  generating the weakly efficient solution set in convex multiobjective
  problems}, SIAM J. Optim., 15 (2005), pp.~987--1001 (electronic).

\bibitem{Mather:1971sh}
{\sc J.~N. Mather}, {\em Stability of {$C^{\infty }$} mappings. {VI}: {T}he
  nice dimensions}, in Proceedings of {L}iverpool {S}ingularities-{S}ymposium,
  {I} (1969/70), Berlin, 1971, Springer, pp.~207--253. Lecture Notes in Math.,
  Vol. 192.

\bibitem{Messac:2004le}
{\sc A.~Messac and C.~A. Mattson}, {\em Normal constraint method with guarantee
  of even representation of complete pareto frontier}, AIAA Journal, 42 (2004),
  pp.~2101--2111.

\bibitem{Messac:2008fi}
{\sc A.~Messac and A.~A. Mullur}, {\em A computationally efficient metamodeling
  approach for expensive multiobjective optimization}, Optimization and
  Engineering, 9 (2008), pp.~37--67.

\bibitem{Michor:1985wb}
{\sc P.~W. Michor}, {\em Elementary catastrophe theory}, vol.~24 of Monografii
  Matematice (Timi\c soara) [Mathematical Monographs (Timi\c soara)],
  Universitatea din Timi\c soara, Facultatea de \c Stiin\c te ale Naturii,
  Sec\c tia Matematic\u a, Timi\c soara, 1985.

\bibitem{Miettinen:1999fk}
{\sc K.~Miettinen}, {\em Nonlinear multiobjective optimization}, International
  Series in Operations Research \& Management Science, 12, Kluwer Academic
  Publishers, Boston, MA, 1999.

\bibitem{Miglierina:2004fr}
{\sc E.~Miglierina and E.~Molho}, {\em {Convergence of minimal sets in convex
  vector optimization}}, SIAM Journal on Optimization, {15} ({2004}),
  pp.~{513--526}.

\bibitem{Miglierina:2008gf}
{\sc E.~Miglierina, E.~Molho, and M.~Rocca}, {\em {Critical points index for
  vector functions and vector optimization}}, Journal of Optimization Theory
  and Applications, {138} ({2008}), pp.~{479--496}.

\bibitem{Milnor:1963fk}
{\sc J.~Milnor}, {\em Morse theory}, Based on lecture notes by M. Spivak and R.
  Wells. Annals of Mathematics Studies, No. 51, Princeton University Press,
  Princeton, N.J., 1963.

\bibitem{Pareto:1896oa}
{\sc V.~Pareto}, {\em Cours d'\'economie politique/Profess\'e \`a
  l'universit\'e de Lausanne}, Rouge, Lausanne, 1896-1897.

\bibitem{Pareto:1906ya}
\leavevmode\vrule height 2pt depth -1.6pt width 23pt, {\em Manuale di economia
  politica con una introduzione alla scienza sociale}, no.~13 in Piccola
  biblioteca scientifica, Societ\`a editrice libraria, Milan, 1906.

\bibitem{Pereyra:2009xz}
{\sc V.~Pereyra}, {\em Fast computation of equispaced pareto manifolds and
  pareto fronts for multiobjective optimization problems}, Mathematics and
  Computers in Simulation, 79 (2009), pp.~1935 -- 1947.
\newblock Applied and Computational MathematicsSelected Papers of the Sixth
  PanAmerican WorkshopJuly 23-28, 2006, Huatulco-Oaxaca, Mexico.

\bibitem{Pinter:2009fb}
{\sc J.~D. Pint{\'e}r}, {\em Global optimization in practice: state of the art
  and perspectives}, in Advances in applied mathematics and global
  optimization, vol.~17 of Adv. Mech. Math., Springer, New York, 2009,
  pp.~377--404.

\bibitem{Porteous:1971kr}
{\sc I.~R. Porteous}, {\em Simple singularities of maps}, in Proceedings of
  {L}iverpool {S}ingularities {S}ymposium, {I} (1969/70), Berlin, 1971,
  Springer, pp.~286--307. Lecture Notes in Math., Vol. 192.

\bibitem{Rakowska:1991ri}
{\sc J.~Rakowska, R.~T. Haftka, and L.~T. Watson}, {\em Tracing the efficient
  curve for multi-objective control-structure optimization}, tech. report,
  Virginia Polytechnic Institute \&\& State University, Blacksburg, VA, USA,
  1991.

\bibitem{Rakowska:1993mq}
\leavevmode\vrule height 2pt depth -1.6pt width 23pt, {\em Multi-objective
  control-structure optimization via homotopy methods}, SIAM J. Optim., 3
  (1993), pp.~654--667.

\bibitem{Rao:1989lg}
{\sc J.~R. Rao and P.~Y. Papalambros}, {\em A non-linear programming
  continuation strategy for one parameter design optimization problems}, in
  Proceedings of ASME Design Automation Conference, Montreal, Quebec, Canada,
  1989, pp.~77--89.

\bibitem{Ruzika:2005dz}
{\sc S.~Ruzika and M.~M. Wiecek}, {\em Approximation methods in multiobjective
  programming}, J. Optim. Theory Appl., 126 (2005), pp.~473--501.

\bibitem{Sacks:1989pz}
{\sc J.~Sacks, W.~J. Welch, T.~J. Mitchell, and H.~P. Wynn}, {\em Design and
  analysis of computer experiments}, Statistical Science, 4 (1989),
  pp.~409--423.

\bibitem{Santner:2003wm}
{\sc T.~J. Santner, B.~J. Williams, and W.~I. Notz}, {\em The design and
  analysis of computer experiments}, Springer Series in Statistics,
  Springer-Verlag, New York, 2003.

\bibitem{Schonlau:2005lp}
{\sc M.~Schonlau and W.~Welch}, {\em Screening the input variables to a
  computer model via analysis of variance and visualization}, Screening Methods
  for Experimentation in Industry, Drug Discovery and Genetics,  (2005),
  pp.~308--327.

\bibitem{Schonlau:1998jh}
{\sc M.~Schonlau, W.~J. Welch, and D.~R. Jones}, {\em Global versus local
  search in constrained optimization of computer models}, New Developments and
  Applications in Experimental Design - IMS Lecture Notes - Monograph Series,
  34 (1998), pp.~11--25.

\bibitem{Schutze:2005fk}
{\sc O.~Sch{\"u}tze, A.~Dell'Aere, and M.~Dellnitz}, {\em On continuation
  methods for the numerical treatment of multi-objective optimization
  problems}, in Practical Approaches to Multi-Objective Optimization,
  J.~Branke, K.~Deb, K.~Miettinen, and R.~E. Steuer, eds., no.~04461 in
  Dagstuhl Seminar Proceedings, Dagstuhl, Germany, 2005, Internationales
  Begegnungs- und Forschungszentrum f{\"u}r Informatik (IBFI), Schloss
  Dagstuhl, Germany.

\bibitem{Sergeyev:2006bu}
{\sc Y.~D. Sergeyev and D.~E. Kvasov}, {\em Global search based on efficient
  diagonal partitions and a set of {L}ipschitz constants}, SIAM J. Optim., 16
  (2006), pp.~910--937 (electronic).

\bibitem{Shewchuk:1996xh}
{\sc J.~R. Shewchuk}, {\em Triangle: Engineering a 2d quality mesh generator
  and delaunay triangulator}, Lecture Notes in Computer Science, 1148 (1996),
  pp.~203--222.

\bibitem{Shewchuk:2002wa}
\leavevmode\vrule height 2pt depth -1.6pt width 23pt, {\em Delaunay refinement
  algorithms for triangular mesh generation}, Computational Geometry: Theory
  and Applications, 22 (2002), pp.~21--74.

\bibitem{Smale:1967p2934}
{\sc S.~Smale}, {\em Differentiable dynamical systems}, Bulletin of the
  American Mathematical Society, 73 (1967), pp.~747--817.

\bibitem{Smale:1969pt}
\leavevmode\vrule height 2pt depth -1.6pt width 23pt, {\em What is global
  analysis?}, Amer. Math. Monthly, 76 (1969), pp.~4--9.

\bibitem{Smale:1973km}
\leavevmode\vrule height 2pt depth -1.6pt width 23pt, {\em Global analysis and
  economics. {I}. {P}areto optimum and a generalization of {M}orse theory}, in
  Dynamical systems ({P}roc. {S}ympos., {U}niv. {B}ahia, {S}alvador, 1971),
  Academic Press, New York, 1973, pp.~531--544.

\bibitem{Smale:1975oy}
\leavevmode\vrule height 2pt depth -1.6pt width 23pt, {\em Optimizing several
  functions}, in Manifolds--{T}okyo 1973 ({P}roc. {I}nternat. {C}onf., {T}okyo,
  1973), Univ. Tokyo Press, Tokyo, 1975, pp.~69--75.

\bibitem{Sobol:2009ph}
{\sc I.~M. Sobol' and S.~Kucherenko}, {\em Derivative based global sensitivity
  measures and their link with global sensitivity indices}, Math. Comput.
  Simulation, 79 (2009), pp.~3009--3017.

\bibitem{Thom:1956nq}
{\sc R.~Thom}, {\em Les singularit{\'e}s des applications diff{\'e}rentiables},
  Ann. Inst. Fourier,  (1956).

\bibitem{Thom:1964lc}
\leavevmode\vrule height 2pt depth -1.6pt width 23pt, {\em
  G{\'e}n{\'e}ralisation de la th{\'e}orie de morse aux vari{\'e}t{\'e}s
  feuillet{\'e}es}, Annales de l'institut Fourier,  (1964).

\bibitem{Thom:1972rw}
\leavevmode\vrule height 2pt depth -1.6pt width 23pt, {\em Stabilit{\'e}
  structurelle et morphog{\'e}n{\`e}se: essai d'une th{\'e}orie
  g{\'e}n{\'e}rale des mod{\`e}les}, Benjamin, New York, Jan 1971.

\bibitem{Thom:1977wq}
\leavevmode\vrule height 2pt depth -1.6pt width 23pt, {\em Structural
  stability, catastrophe theory, and applied mathematics: The {John von Neumann
  Lecture}, 1976}, SIAM Review, 19 (1977), pp.~189--201.

\bibitem{Utyuzhnikov:2009yj}
{\sc S.~V. Utyuzhnikov, P.~Fantini, and M.~D. Guenov}, {\em A method for
  generating a well-distributed {P}areto set in nonlinear multiobjective
  optimization}, J. Comput. Appl. Math., 223 (2009), pp.~820--841.

\bibitem{Wan:1975bf}
{\sc Y.~H. Wan}, {\em Morse theory for two functions}, Topology, 14 (1975),
  pp.~217--228.

\bibitem{Wan:1975lt}
\leavevmode\vrule height 2pt depth -1.6pt width 23pt, {\em On local {P}areto
  {O}ptima}, J. Math. Econom., 2 (1975), pp.~35--42.

\bibitem{Wan:1978jy}
\leavevmode\vrule height 2pt depth -1.6pt width 23pt, {\em On the structure and
  stability of local {P}areto optima in a pure exchange economy}, J. Math.
  Econom., 5 (1978), pp.~255--274.

\bibitem{Watson:1981kq}
{\sc D.~F. Watson}, {\em Computing the {$n$}-dimensional {D}elaunay
  tessellation with application to {V}orono\u\i\ polytopes}, Comput. J., 24
  (1981), pp.~167--172.

\bibitem{Wu:2000vo}
{\sc C.~F.~J. Wu and M.~S. Hamada}, {\em Experiments: Planning, Analysis, and
  Parameter Design Optimization}, John Wiley, New York, 2000.

\bibitem{Wu:2005lp}
{\sc Y.~Wu, L.~Ozdamar, and A.~Kumar}, {\em Triopt: a triangulation-based
  partitioning algorithm for global optimization}, Journal of Computational and
  Applied Mathematics, 177 (2005), p.~35.

\end{thebibliography}
\end{document}